\documentclass[10pt,a4paper]{imsart}

\RequirePackage[OT1]{fontenc}
\RequirePackage{amsthm,amsmath}
\RequirePackage[authoryear]{natbib}
\RequirePackage[colorlinks,citecolor=blue,urlcolor=blue]{hyperref}
\RequirePackage{hypernat}

\newcommand{\ignore}[1]{}{}

\newcommand{\eq}{\eqref}

\def\be#1{\begin{equation*}#1\end{equation*}}
\def\ben#1{\begin{equation}#1\end{equation}}
\def\bes#1{\begin{equation*}\begin{split}#1\end{split}\end{equation*}}
\def\besn#1{\begin{equation}\begin{split}#1\end{split}\end{equation}}

\newcommand{\beqn}{\begin{eqnarray}}             
\newcommand{\eeqn}{\end{eqnarray}}               
\newcommand{\beq}{\begin{eqnarray*}}             
\newcommand{\eeq}{\end{eqnarray*}}               

\newcommand{\nn}{\nonumber}

\usepackage{graphicx}
\usepackage{latexsym}
\usepackage{amsmath,amsthm,amssymb,amscd}
\usepackage{epsf,amsmath}

\def\E{{\mathbb{E}}}
\def\Var{{\rm Var}}
\def\cov{{\rm Cov}}

\def\P{{\mathbb{P}}}

\def \min{{\rm min}}

\startlocaldefs
\numberwithin{equation}{section}
\theoremstyle{plain}
\newtheorem{thm}{Theorem}[section]

\newtheorem{lem}[thm]{Lemma}

\newtheorem{cor}[thm]{Corollary}
\newtheorem{prop}[thm]{Proposition}
\newtheorem{rem}[thm]{Remark}
\newtheorem{exa}[thm]{Example}
\endlocaldefs

\begin{document}

\begin{frontmatter}

\title{Multivariate Normal Approximation by Stein's Method: The Concentration Inequality Approach\protect}
\runtitle{Multivariate Normal Approximation by Stein's Method}

\begin{aug}
\author{\fnms{Louis H.Y.} \snm{Chen}
\ead[label=e1]{matchyl@nus.edu.sg}}
\and
\author{\fnms{Xiao} \snm{Fang}
\ead[label=e2]{stafx@nus.edu.sg}}

\runauthor{L.H.Y. Chen AND X. Fang}

\affiliation{National University of Singapore}

\address{Department of Mathematics\\
National University of Singapore\\
10, Lower Kent Ridge Road\\
Singapore 119076\\
Republic of Singapore\\
\printead{e1}}

\address{Department of Statistics and Applied Probability\\
National University of Singapore\\
6 Science Drive 2\\
Singapore 117546\\
Republic of Singapore\\
\printead{e2}}
\end{aug}

\begin{abstract}
The concentration inequality approach for normal approximation by Stein's method is generalized to the multivariate setting. We use this approach to prove a non-smooth function distance for multivariate normal approximation for standardized sums of $k$-dimensional independent random vectors $W=\sum_{i=1}^n X_i$
with an error bound of order $k^{1/2}\gamma$ where $\gamma=\sum_{i=1}^n \E|X_i|^3$.
For sums of locally dependent (unbounded) random vectors, we obtain a fourth moment bound which is typically of order $O_k(1/\sqrt{n})$, as well as a third moment bound which is typically of order $O_k(\log n/\sqrt{n})$.
\end{abstract}



\end{frontmatter}


\section{Introduction}

Since Stein introduced his method for normal approximation in 1972, much has been developed for normal approximation in one dimension for dependent random variables for both smooth and non-smooth functions. A typical non-smooth function is the indicator of a half line.  Three approaches have been developed to deal with non-smooth functions: the induction approach popularized by \cite{Bo84}, the recursive approach of \cite{Ra03} and the concentration inequality approach developed by \cite{Ch86, Ch98}, \cite{ChSh01, ChSh04}.

Although Stein's method has been extended to multivariate normal approximation (see, for example, \cite{Ba90}, \cite{Go91}, \cite{GoRi96}, \cite{ChMe08}, \cite{ReRo09}), relatively few results have been obtained for non-smooth functions, typically for indicators of convex sets in finite dimensional Euclidean spaces. In general, it is much harder to obtain optimal bounds for non-smooth functions than for smooth functions. As far as we know, results for non-smooth functions are those of \cite{Go91}, \cite{RiRo96} and \cite{BhHo10}, which is an exposition of G\"otze's result. While the result of \cite{RiRo96} is for bounded locally dependent random vectors, those of \cite{Go91} and of \cite{BhHo10} are for independent random vectors with finite third moments.  The approach of \cite{Go91} and of \cite{BhHo10} is by induction.

In this paper, we extend the concentration inequality approach to the multivariate setting. We prove that for $W=\sum_{i=1}^n X_i$ being a sum of independent random vectors, standardized to have $0$ mean and identity covariance matrix,
\besn{
\P(W^{(i)}\in A^{4\gamma+\epsilon}\backslash A^{4\gamma})\le 4.1 k^{1/2}\epsilon+39k^{1/2}\gamma
}
and with $|\cdot|$ denoting the Euclidean norm of a vector,
\besn{
\P(W\in A^{4\gamma+|X_i|}\backslash A^{4\gamma})\le 4.1 k^{1/2}\E |X_i|+39k^{1/2}\gamma
}
where $A$ is a convex set in $\mathbb{R}^k$, $A^\epsilon=\{x\in\mathbb{R}^k: d(x, A)\le \epsilon\}$ for $\epsilon>0$, $W^{(i)}=W-X_i$ and $\gamma=\sum_{i=1}^n\E |X_i|^3$. Using these concentration inequalities, we prove a normal approximation theorem for $W$ with an error bound of the order $k^{1/2}\gamma$. This dependence of $k^{1/2}$ on the dimension is better than $k^{5/2}$ and $k^{3/2}$ obtained by \cite{BhHo10} and $k$ as stated in \cite{Go91}.
Comparing our result with those assuming finite third moments and using other methods in the literature, only the result of \cite{Be05} gives a bound depending on $k^{1/4}$, which is better than $k^{1/2}$. Other results for i.i.d. random vectors, for example, by \cite{Na76},  \cite{Se80} and \cite{Sa81} depend on $k$. 

Our concentration inequality approach provides a new way of dealing with dependent random vectors, for example, those under local dependence, for which the induction approach is not likely to be applicable. In Section 4, we prove two multivariate normal approximation theorems for sums of locally dependent random vectors assuming finite fourth and third moments and giving error bounds typically of order $O_k(1/\sqrt{n})$ and $O_k(\log n/\sqrt{n})$ respectively. We apply them to problems with graph dependence structure and the joint distribution of sums of partial products in a sequence of independent and identically distributed random variables. 

The paper is organized as follows. In section 2, we develop techniques for the concentration inequality approach in the multivariate setting. In sections 3 and 4, we use the concentration inequality approach to obtain multivariate normal approximation theorems for sums of independent and locally dependent random vectors. In Section 5, we prove the results for local dependence. In section 6, we prove the technical lemmas in Section 2.

Throughout the paper, let $|\cdot|$ denote the Euclidean norm of a vector or the cardinality of a set, and let $||\cdot||$ denote the operator norm of a matrix. 
For a real-valued function $f$ on $\mathbb{R}^k$, we will write $\partial_j f(x)$ for $\partial f(x)/\partial x_j$, $\partial_{j j_1} f(x)$ for $\partial^2 f(x)/(\partial x_{j} \partial x_{j_1})$ and so on.
Let $a\cdot b$ denote the inner product of two vectors.
For convenience, sometimes we will use $\E^X Y$ for $\E(Y|X)$ and use $\P^X(A)$ for $\P(A|X)$.
For a positive integer $k$, let $[k]=\{1,2,\ldots,k\}$. Finally, let $I_{k\times k}$ denote the $k$ by $k$ identity matrix.

\section{Concentration inequalities}

As a powerful tool of proving distributional approximations along with error bounds, the theory of Stein's method has been extensively developed in the literature for random variables with all kinds of dependence structure. While it works well for smooth function distances, it requires much more efforts to obtain optimal bounds for non-smooth function distances such as the Kolmogorov distance. To overcome this difficulty, we consider the probability for some random variable $W$ taking values in a small interval $[a, b]$. A bound on ${\rm P}(W\in [a, b])$ is called a concentration inequality. Now if $W$ is a $k$-dimensional random vector and $Z$ is a $k$-dimensional standard Gaussian random vector, the non-smooth function distance between $\mathcal{L}(W)$ and $\mathcal{L}(Z)$ usually means $\sup_{A\in \mathcal{A}} |{\rm P}(W\in A)-{\rm P}(Z\in A)|$ where $\mathcal{A}$ denotes the set of all convex sets in $\mathbb{R}^k$. A concentration inequality in this setting would be a bound on ${\rm P}(W\in A^\epsilon\backslash A)$ where $A^\epsilon=\{x\in \mathbb{R}^k: d(x, A)\le \epsilon\}$  where $d(x, A)=\inf_{y\in A} |x-y|$.

For a given convex set $A\subset \mathbb{R}^k$, $\epsilon>0$, we define $f=f(A,\epsilon)=(f_1, f_2, \ldots, f_k)^t: \mathbb{R}^k\rightarrow \mathbb{R}^k$ as follows. For $x\in \bar{A}$ where $\bar{A}$ is the closure of $A$, $f(x)=0$. For $x\in A^{\epsilon}\backslash \bar{A}$, find $x_0$ the nearist point in $\bar{A}$ from $x$, and define $f(x)=x-x_0$. For $x\in \mathbb{R}^k\backslash A^{\epsilon}$, find $x_0$ the nearist point in $\bar{A}$ from $x$, and $x_1$ the intersection of $\{x_0+t(x-x_0): t\in [0,1]\}$ and $\partial A^{\epsilon}$, the boundary of $A^\epsilon$, and define $f(x)=x_1-x_0=f(x_1)$. We have the following four lemmas regarding the properties of the above defined $f$.

\begin{lem}\label{lem1}
We have
\besn{\label{lem1-1}
|f|\le \epsilon.
}
\end{lem}

\begin{lem}\label{lem3}
For all $\xi,\eta \in \mathbb{R}^k$,
\besn{\label{lem1-2}
\xi \cdot (f(\eta+\xi)-f(\eta)) \ge 0.
}
\end{lem}

\begin{lem}\label{lem2}
For every $i\in [k]$ and any fixed $x_1,\ldots, x_{i-1},x_{i+1},\ldots, x_k$, $f_i$ is absolutely continuous in $x_i$ and
\besn{
\partial_i f_i(x) \ge 0\quad  \text{a.e.} .
}
\end{lem}

For $x\in (A^{\epsilon})^o \backslash \bar{A}$, where $A^o$ is the interior of $A$, we have a shaper lower bound for $\partial_i f_i(x)$. Let $\theta=(\theta_1,\theta_2,\ldots,\theta_k)^t$ be the angles between $x-x_0$ and the axes.

\begin{lem}\label{lem4}
For all $i\in [k]$, $x\in (A^{\epsilon})^o \backslash \bar{A}$,
\besn{\label{lem1-3}
\partial_i f_i(x) \ge \cos ^2 \theta_i \quad \text{a.e.}.
}
\end{lem}

We defer the proofs of the lemmas to Section 4. To obtain a concentration inequality for a random vector $W$ of interest, we apply the above defined function $f$ in the Stein identity for $W$. We first derive a concentration inequality for multivariate Gaussian vectors, then derive a general concentration inequality and apply it to sums of independent and locally dependent random vectors.

\subsection{Multivariate normal distribution}

\begin{prop}\label{p1}
Let $Z=(Z_1,Z_2,\ldots,Z_k)^t$ be a $k$-dimensional standard
Gaussian random vector. Then for any convex set $A$ in $\mathbb{R}^k$ and
$\epsilon_1, \epsilon_2\ge 0$,
\besn{\label{p1-1}
\P (Z\in A^{\epsilon_1}\backslash A^{-\epsilon_2})\le k^{1/2}(\epsilon_1+\epsilon_2)
}
where $A^\epsilon=\{x\in \mathbb{R}^k: d(x, A)\le \epsilon\}$ and $A^{-\epsilon}=\{x\in \mathbb{R}^k: B(x,\epsilon)\subset A\}$ where $B(x, \epsilon)$ is the $k$-dimensional ball centered in $x$ with radius $\epsilon$.
\end{prop}

\begin{proof}
From the joint independence among $\{Z_1,Z_2,\ldots,Z_k\}$ and the
integration by parts formula, we have the following $k$ functional
identities for $Z$.
\besn{
\E Z_1f_1(Z)&=\E \partial_1 f_1(Z), \\
&\cdots \\
\E Z_kf_k(Z)&=\E \partial_k f_k(Z).
}

Using the function $f=f(A, \epsilon)$ defined at the beginning of this section where $A$ is a convex set in $\mathbb{R}^k$ and $\epsilon>0$ and summing up the above $k$ equations, we have
\ben{\label{p1.0}
\sum_{j=1}^k \E Z_jf_j(Z)=\sum_{j=1}^k \E \partial_j f_j(Z).
}
By Lemma \ref{lem1}, LHS of \eq{p1.0}$\le \epsilon\E|Z| = k^{1/2} \epsilon$. By Lemma \ref{lem2} and Lemma \ref{lem4},
\besn{
\text{RHS of}\  \eq{p1.0}&\ge \sum_{j=1}^k \E \partial_j f_j(Z) I(Z\in (A^\epsilon)^o \backslash \overline{A}) \\
&\ge \E \sum_{j=1}^k \cos^2 \theta_j I(Z\in (A^\epsilon)^o \backslash \overline{A})=\P (Z\in (A^\epsilon)^o \backslash \overline{A}).
}
Therefore,
\besn{\label{p1-2}
\P(Z\in A^\epsilon\backslash A)\le k^{1/2} \epsilon.
}
The bound (\ref{p1-1}) can be deduced from \eq{p1-2} by the arguments in Section 1.3 of \cite{BhRa86} sketched as follows.

Without loss of generality, assume $A^o \ne \emptyset$. First suppose $A$ is bounded. Given any $\delta>0$, we may choose $x_1,x_2,\ldots, x_n \in \partial A$ such that $\partial A \subset \{x_1,\ldots, x_n\}^\delta$. Let $P$ be the convex hull of $\{x_1,\ldots, x_n\}$. By taking $\delta$ small enough, $P^o\ne \emptyset$. For some positive integer $m$, $P$ can be expressed as
\beq
P=\{x\in \mathbb{R}^k: u_j \cdot x\le d_j, 1\le j\le m\}
\eeq
where $u_j$'s are distinct unit vectors and $d_j$'s are real numbers. For each real $a$, define
\beq
P_a=\{x\in \mathbb{R}^k: u_j \cdot x\le d_j+a, 1\le j\le m\}.
\eeq
Then from the fact that $P\subset A \subset P^\delta$, we have
\beq
A^{\epsilon_1}\backslash A^{-\epsilon_2} \subset (P^\delta)^{\epsilon_1}\backslash P_{-\epsilon_2}\subset P_{\epsilon_1+\delta}\backslash P_{-\epsilon_2}.
\eeq
Therefore,
\besn{\label{p1-3}
\P (Z\in A^{\epsilon_1}\backslash A^{-\epsilon_2}) \le \P(Z\in P_{\epsilon_1+\delta}\backslash P_{-\epsilon_2})=\int_{-\epsilon_2}^{\epsilon_1+\delta} \int_{\partial P_a} \phi d\lambda_{k-1}da
}
where $\phi$ is the density of standard $k$-dimensional normal distribution and $\lambda_{k-1}$ is the Lebesgue measure in $\mathbb{R}^{k-1}$. We used Lemma 3.9 in \cite{BhRa86} in the last equality. From the arguments leading to (3.35) in \cite{BhRa86},
\beq
|\P(Z\in (P_a)^\epsilon \backslash P_a)-\epsilon \int_{\partial P_a} \phi d\lambda_{k-1} |\le o(\epsilon), \ \text{as}\ \epsilon \rightarrow 0.
\eeq
The above inequality and (\ref{p1-2}) result in
\beq
\int_{\partial P_a} \phi d\lambda_{k-1} \le k^{1/2}.
\eeq
Therefore, from (\ref{p1-3}),
\beq
\P(Z\in A^{\epsilon_1}\backslash A^{-\epsilon_2})\le k^{1/2} (\epsilon_1+\epsilon_2+\delta).
\eeq
The bound (\ref{p1-1}) is proved by letting $\delta\rightarrow 0$. If $A$ is unbounded, consider $A_r=A\cap B(0,r)$ and let $r\rightarrow \infty$.
\end{proof}

\begin{rem}
It is known that $\P (Z \in A^{\epsilon_1} \backslash A^{-\epsilon_2})\le 4k^{1/4} (\epsilon_1+\epsilon_2)$, which is of optimal order in $k$ (see \cite{Ba93} and \cite{Be03}). It is not clear how we can obtain $k^{1/4}$ in the bound by our approach.
\end{rem}

\subsection{General concentration inequalities}

\begin{prop}\label{p3.1}
Let $(W,W')$ be an exchangeable pair ($\mathcal{L}(W, W')=\mathcal{L}(W', W)$) of $k$-dimensional random vectors. Let $D=(D_1,\dots, D_k)^t=W'-W$. Suppose
\be{
\E |D|^3 < \infty,\quad \inf_{\xi\in S^{k-1}} \E (D\cdot \xi)^2>0
}
where $S^{k-1}$ denotes the unit $(k-1)$-dimensional sphere.
Define
\ben{\label{p3.1-0}
\delta:={2\E |D|^3 \over \inf_{\xi\in S^{k-1}} \E (D\cdot \xi)^2}.
}
Then for any convex set $A$ in $\mathbb{R}^k$ and any $\epsilon>0$, we have
\besn{\label{p3.1-1}
\P(W\in A^{\epsilon+\delta}\backslash A^\delta)&\leq \frac{1}{\inf_{\xi\in S^{k-1}} \E (D\cdot \xi)^2} \Bigg\{ 
\frac{16}{3}(\epsilon+2\delta)\sqrt{\sum_{j=1}^k \Var[\E(D_j|W)]}\\
&\kern4em + 2\sqrt{\sum_{j,j_1=1}^k \Var\left\{ \E \left[D_j D_{j_1}I(|D|\leq \delta)|W \right] \right\}}
 \Bigg\}.
}
\end{prop}
\begin{rem}
For $\delta_1\geq \delta$, $A^{\epsilon+\delta_1}\backslash A^{\delta_1}=(A^{\delta_1-\delta})^{\epsilon+\delta}\backslash (A^{\delta_1-\delta})^{\delta}$. Therefore, \eq{p3.1-1} remains true if we replace $\delta$ on the left-hand side by any constant $\delta_1\geq \delta$.
A $1$-dimensional version of the above proposition can be found in Lemma 2.1 of \cite{ChFa13}. As in the $1$-dimensional case, truncating $|D|$ at $\delta$ allows us to keep within third moments.
Also we do not require $\E(D|W)$ to be approximately linear in $W$ as usually assumed in the literature on Stein's method of exchangeable pairs. Therefore, we are able to apply Proposition \ref{p3.1} to the local dependence case.
\end{rem}
\begin{rem}
For the concentration inequality to be useful, the denominator $\inf_{\xi\in S^{k-1}} \E (D\cdot \xi)^2$ in\eq{p3.1-1} should not be too small, in other words, $W'$ needs to be different from $W$ in every direction. For technical reasons, we are only able to bound $\P(W\in A^{\epsilon+\delta}\backslash A^\delta)$ instead of $\P(W\in A^\epsilon\backslash A)$. This does not affect the final bound in multivariate normal approximation when $\delta$ is small.
\end{rem}
\begin{proof}[Proof of Proposition \ref{p3.1}]
We use $f=f(A, \epsilon+2\delta)$ defined at the beginning of this section in the Stein identity for $W$
\be{
\E (W'-W) \cdot (f(W')+f(W))=0,
}
which follows from the exchangeability of $(W,W')$ and implies
\ben{\label{60}
-2 \E (W'-W)\cdot f(W)=\E (W'-W)\cdot (f(W')-f(W)).
}
Because $|f|\le \epsilon+2\delta$ by Lemma \ref{lem1},
\besn{\label{p3.11}
\text{LHS of (\ref{60})} &\le 2(\epsilon+2\delta)\E|\E(W'-W|W)|\\
&\leq 2(\epsilon+2\delta)\sqrt{\E \sum_{j=1}^k [\E(D_j|W)]^2}\\
&=2(\epsilon+2\delta)\sqrt{\sum_{j=1}^k \Var[\E(D_j|W)]}.
} 
From Lemma \ref{lem3},
\bes{
&\text{RHS of (\ref{60})}\nn \\
&\ge  \E D \cdot (f(W')-f(W))I(|D|\le \delta)I(W\in A^{\epsilon+\delta} \backslash A^{\delta}) \\
&=  \E \big\{  \sum_{j=1}^2 (D \cdot H_{j})  ( f(W') \cdot H_{j} -f(W)\cdot H_{j} )   \big\}    I(|D|\le \delta) I(W \in A^{\epsilon+\delta} \backslash A^{\delta})
}
where we used the (random) orthonormal basis $\{H_{1},\ldots, H_{k}\}$ defined as follows. For each $W=w\in A^{\epsilon+\delta} \backslash A^{\delta}$ and $D=d$, define an orthonormal basis $\{H_{1},\ldots, H_{k}\}=\{h_{1},\ldots, h_{k}\}$ such that $h_{1}$ and $w-w_0$ are parallel and $h_{2}$ and $d-(d \cdot h_{1})h_{1}$ are parallel ($0$-vector is parallel to any vector). Recall that $w_0$ is the nearest point in $\bar{A}$ from $w$. Then,
\bes{
\text{RHS of (\ref{60})}&\ge  \E \biggl\{ (D \cdot H_{1})(f(W+(D \cdot H_{1}) H_{1})\cdot H_{1}-f(W) \cdot H_{1}) \\ 
&\qquad+ (D \cdot H_{1}) (f(W+D)\cdot H_{1}-f(W+(D \cdot H_{1})H_{1})\cdot H_{1}) \\
&\qquad+ (D \cdot H_{2}) (f(W+(D \cdot H_{1}) H_{1})\cdot H_{2}-f(W) \cdot H_{2}) \\
&\qquad +(D \cdot H_{2} ) (f(W+D)\cdot H_{2}-f(W+(D \cdot H_{1})H_{1})\cdot H_{2}) \biggr\} \\
&\kern6em \times I(|D|\le \delta) I(W\in A^{\epsilon+\delta} \backslash A^{\delta}).
}
If $w\in A^{\epsilon+\delta} \backslash A^{\delta}, |d|\le \delta$, then we have
\besn{\label{e1}
f(w+(d \cdot h_{1})h_{1})\cdot h_{1}-f(w)\cdot h_{1}=d \cdot h_{1},
}
\besn{\label{e2}
f(w+(d \cdot h_{1} )h_{1})\cdot h_{2}-f(w)\cdot h_{2}=0
}
and
\besn{\label{e3}
&(d\cdot h_{2})(f(w+d)\cdot h_{2}-f(w+(d \cdot h_{1})h_{1})\cdot h_{2}) \\
&\ge (f(w+d)\cdot h_{1}-f(w+(d\cdot h_{1})h_{1})\cdot h_{1})^2.
}
Equations (\ref{e1}) and (\ref{e2}) follow from $f(w+(d\cdot h_{1})h_{1})=f(w)+(d\cdot h_{1})h_{1}$. For (\ref{e3}), consider the plane $p$ parallel to $h_{1}, h_{2}$ and containing $w$. Let $l$ be the line parallel to $h_{2}$ and containing $w_0$. The line $l$ divides $p$ into two parts $p_1, p_2$ where $p_1$ is closed and $p_2$ is open and contains $w$. Draw a circle on $p$ with diameter $[w_0, w+d]$. Then $(w+d)'$, the projection of $(w+d)_0$ on $p$, must be inside the circle (or on the perimeter) and on $p_1$ because of the convexity of $A$. Let $(w+d)''$ be the projection of $w+d$ on $l$, and let $(w+d)'''$ be the projection of $(w+d)'$ on $l$. Then, \eq{e3} follows from
\bes{
|((w+d)''-w_0)((w+d)'''-w_0)|\ge |(w+d)'-(w+d)'''|^2,
}
which is a consequence of the fact that the angle between $(w+d)''-(w+d)'$and $w_0-(w+d)'$ is greater than or equal to $\pi/2$. Using $ab\ge -a^2-b^2/4$,
\besn{\label{e4}
&(d\cdot h_{1})(f(w+d)\cdot h_{1}-f(w+(d \cdot h_{1})h_{1})\cdot h_{1})\\
&\ge -\frac{(d\cdot h_{1})^2}{4}-(f(w+d)\cdot h_{1} -f(w+(d \cdot h_{1})h_{1})\cdot h_{1})^2.
}
Applying \eq{e1}-\eq{e4}, we obtain a lower bound of RHS of (\ref{60}) as
\besn{\label{56}
\text{RHS of (\ref{60})} \ge \frac{3}{4} \E (D\cdot H_{1})^2 I(|D|\le \delta) I(W\in A^{\epsilon+\delta} \backslash A^{\delta}).
}
In other words, we have
\besn{
\text{RHS of (\ref{60})}
&\ge \frac{3}{4} \E (D \cdot \xi(W))^2I(|D|\le \delta) I(W\in A^{\epsilon+\delta} \backslash A^{\delta})\\
&=:R
}
where $\xi(W)=(W_0-W)/|W_0-W|$ for $W\in A^{\epsilon+\delta} \backslash A^{\delta}$ and $W_0$ is the nearist point in $\bar{A}$ from $W$. We may define $\xi(W)$ to be $e_1$, where $\{e_1,\ldots, e_k\}$ are the original orthonormal basis when $W\notin A^{\epsilon+\delta} \backslash A^{\delta}$, since it does not affect the value of $R$. We now derive a lower bound of $R$.
\bes{
R
&=\frac{3}{4} \E \sum_{j=1}^k D_{j}^2 \xi(W)_{j}^2 I(|D|\le \delta)I(W\in A^{\epsilon+\delta} \backslash A^{\delta}) \\
&\quad+\frac{3}{4} \E \sum_{j\ne j_1} D_{j}D_{j_1} \xi(W)_{j}\xi(W)_{j_1} I(|D|\le \delta)I(W\in A^{\epsilon+\delta} \backslash A^{\delta}) \\
&=:R_1+R_2.
}
For $R_1$,
\bes{
R_1
&= \frac{3}{4} \sum_{j=1}^k \E I(W\in A^{\epsilon+\delta} \backslash A^{\delta})\xi(W)_{j}^2 D_{j}^2 I(|D|\le \delta)\\
&=\frac{3}{4} \sum_{j=1}^k \E I(W\in A^{\epsilon+\delta} \backslash A^{\delta})\xi(W)_{j}^2 \left[ D_{j}^2 I(|D|\le \delta)-\E D_{j}^2 I(|D|\le \delta)\right]\\
&\quad+\frac{3}{4} \sum_{j=1}^k \E I(W\in A^{\epsilon+\delta} \backslash A^{\delta})\xi(W)_{j}^2 \E D_{j}^2 I(|D|\le \delta)\\
&=: R_{1,1}+R_{1,2}.
}
Using the inequality
\besn{\label{61}
ab\le \frac{\delta\theta}{4} a^2+\frac{b^2}{\delta\theta}
}
for a positive $\theta$ to be chosen,
\bes{
|R_{1,1}|&\le \frac{3}{4}\sum_{j=1}^k\biggl\{\frac{\delta\theta}{4} \E  \xi(W)_{j}^4 +\frac{1}{\delta\theta}\E \Bigl[ \E[D_{j}^2 I(|D|\le \delta)|W]-\E D_{j}^2 I(|D|\le \delta) \Bigr]^2\biggr\} \\
&= \frac{3}{4}\big\{ \frac{\delta\theta}{4} \sum_{j=1}^k \E \xi(W)_{j}^4 +\frac{1}{\delta\theta} \sum_{j=1}^k \Var[\E(D_{j}^2 I(|D|\le \delta)|W)]  \big\}.
}
We write $R_{1,2}$ as
\be{
R_{1,2}=\frac{3}{4} \sum_{j=1}^k \E I(W\in A^{\epsilon+\delta} \backslash A^{\delta})\xi(W)_{j}^2 [\E D_{j}^2-\E D_{j}^2 I(|D|> \delta)].
}
Similarly,
\bes{
R_2&=\frac{3}{4} \E \sum_{j\ne j_1} D_{j}D_{j_1} \xi(W)_{j}\xi(W)_{j_1} I(|D|\le \delta)I(W\in A^{\epsilon+\delta} \backslash A^{\delta})\\
&=\frac{3}{4}\sum_{j\ne j_1}\E I(W\in A^{\epsilon+\delta} \backslash A^{\delta})\xi(W)_{j}\xi(W)_{j_1} \\
&\qquad \times(D_{j}D_{j_1}I(|D|\le \delta)-\E D_{j}D_{j_1}I(|D|\le \delta))\\
&\quad+\frac{3}{4}\sum_{j\ne j_1}\E I(W\in A^{\epsilon+\delta} \backslash A^{\delta})\xi(W)_{j}\xi(W)_{j_1}\E D_{j}D_{j_1}I(|D|\le \delta)\\
&=:R_{2,1}+R_{2,2}
}
and
\be{
|R_{2,1}|\le  \frac{3}{4}\big\{ \frac{\delta\theta}{4} \sum_{j\ne j_1} \E [\xi(W)_{j} \xi(W)_{j_1}]^2 + \frac{1}{\delta\theta}\sum_{j\ne j_1}\Var[ \E (D_{j}D_{j_1}I(|D|\le \delta) |W)] \big\},
}
\be{
R_{2,2}=\frac{3}{4}\sum_{j\ne j_1}\E I(W\in A^{\epsilon+\delta} \backslash A^{\delta})\xi(W)_{j}\xi(W)_{j_1}\big[ \E D_{j}D_{j_1} - \E D_j D_{j_1}I(|D|> \delta) \big].
}
From the bounds on $|R_{1,1}|$ and $|R_{2,1}|$ and $|\xi(W)|=1$,
\bes{
|R_{1,1}|+|R_{2,1}|\le \frac{3\delta\theta}{16}+\frac{3}{4\delta\theta}\sum_{j,j_1=1}^k \Var[ \E (D_{j}D_{j_1}I(|D|\le \delta) |W)].
}
Also, with $\widetilde{W}$ being an independent copy of $W$,
\bes{
R_{1,2}+R_{2,2}&\geq \frac{3}{4} \E I(\widetilde{W}\in A^{\epsilon+\delta}\backslash A^{\delta})(\xi(\widetilde{W})\cdot D)^2-\frac{3 \E |D|^3}{4\delta} \P(W\in A^{\epsilon+\delta}\backslash A^{\delta})\\
&\geq \frac{3}{4}\P(\widetilde{W}\in A^{\epsilon+\delta}\backslash A^{\delta})(\inf_{\xi\in S^{k-1}} \E (D\cdot \xi)^2-\frac{\E|D|^3}{\delta}). 
}
The last two inequalities, along with \eq{p3.11}, yield
\bes{
&2(\epsilon+2\delta)\sqrt{\sum_{j=1}^k \Var[\E(D_j|W)]}+\frac{3\delta\theta}{16}+\frac{3}{4\delta\theta}\sum_{j,j_1=1}^k \Var[ \E (D_{j}D_{j_1}I(|D|\le \delta) |W)]\\
&\geq \frac{3}{4} \P (W\in A^{\epsilon+\delta}\backslash A^{\delta})(\inf_{\xi\in S^{k-1}} \E (D\cdot \xi)^2 - \frac{\E|D^3|}{4\delta}).
}
Proposition \ref{p3.1} is proved by recalling the definition of $\delta$ \eq{p3.1-0} and choosing
\be{
\theta=\frac{2}{\delta}\sqrt{\sum_{j,j_1=1}^k \Var[ \E (D_{j}D_{j_1}I(|D|\le \delta) |W)]}.
}
\end{proof}
Now we consider $W=\sum_{i=1}^n X_i$ being a sum of locally dependent $k$-dimensional random vectors. To avoid confusion, for a set of $n$ $k$-dimensional vectors, we always use $i,i_1,\dots\in [n]$ to index them, and use $j,j_1,\dots\in [k]$ to index their components. Assume (LD3), i.e., for each $i\in [n]$, there exist neighborhoods $A_i$, $B_i$, $C_i\subset [n]$ such that $X_i$ is independent of $\{X_{i_1}: i_1\notin A_i\}$, $\{X_{i_1}: i_1\in A_i\}$ is independent of $\{X_{i_1}: {i_1}\notin B_i\}$, and $\{X_{i_1}: i_1\in B_i\}$ is independent of $\{X_{i_1}: i_1\notin C_i\}$. For such a $W$, an exchangeable pair $(W,W')$ can be constructed as follows. 
Let $\{X_1^*,\dots,X_n^*\}$ be an independent copy of $\{X_1,\dots, X_n\}$ and, for $i=1,\dots, n$, let $\{X_{i_1i}': i_1\in A_i\backslash \{i\}\}$ and $\{X_{i_1}: i_1\in A_i\backslash \{i\}\}$ be conditionally independent given $\{X_{i_1}: i_1\notin A_i\}$ such that
\bes{
&\mathcal{L}\{X_{i_1i}': i_1\in A_i\backslash \{i\}\big| X_i^*=x_i, X_{i_1}=x_{i_1}\  \text{for}\  i_1\notin A_i\}\\
&=\mathcal{L}\{X_{i_1}: i_1\in A_i\backslash \{i\}\big| X_i=x_i, X_{i_1}=x_{i_1} \ \text{for}\  i_1\notin A_i\}.
}
Let $I$ be a uniform random index from $[n]$ and independent of all the other random vectors. Define
\be{
W'=W^I=W-\sum_{i_1\in A_I} X_{i_1}+X_I^*+\sum_{i_1\in A_I\backslash \{I\}} X_{i_1I}'.
}
Then $(W,W')$ is an exchangeable pair.
\begin{cor}\label{c3.3}
Let $W=\sum_{i=1}^n X_i$ be a sum of $k$-dimensional random vectors such that $\E X_i=0$ for each $i\in [n]$. Assume (LD3) with neighborhood size bounded by
\besn{\label{c3.3--1}
&|A_i|,\  \max_{i\in [n]}|\{i_1: i\in A_{i_1}\}|\leq \theta_1,\\
&\max_{i\in [n]}|\{i_1: B_{i_1}\cap C_i\ne \emptyset\}|,\ 
\max_{i\in [n]}|\{i_1: B_{i}\cap C_{i_1}\ne \emptyset\}|  \leq \theta_2.
}
Let the exchangeable pair $(W,W')$ be constructed as above and let $D=D^I=W'-W$. Define
\besn{\label{c3.3-0}
&\alpha:=2\inf_{\xi\in S^{k-1}} \sum_{i=1}^n \E [\Var(Y_i\cdot \xi|X_{B_i\backslash A_i})],\\
&\beta:=2\sum_{j=1}^k \sum_{i=1}^n \E [\Var(Y_{ij}|X_{B_i\backslash A_i})],\\
&\gamma:=\sum_{i=1}^n \E |X_i|^3,\quad \delta=\frac{16\theta_1 \gamma}{\alpha}
}
where 
\ben{\label{c3.3-2}
Y_{i}:=\sum_{i_1\in A_i} X_{i_1},\quad X_{B_i\backslash A_i}:=\{X_{i_1}: i_1\in B_i\backslash A_i\}.
}
Then, for any convex set $A$ in $\mathbb{R}^k$ and any $\epsilon>0$,
\ben{\label{c3.3-1}
\P(W\in A^{\epsilon+\delta}\backslash A^{\delta})\leq \frac{\theta_1^3 \sqrt{\theta_2}}{\alpha^{3/2}}(16\sqrt{2}+\frac{512}{3}\sqrt{\frac{\beta}{\alpha}})\gamma+\frac{16\sqrt{\theta_2 \beta}}{3\alpha}\epsilon.
}
\end{cor}
\begin{rem}
Assuming $\E W W^t=I_{k\times k}$, we have $\forall\ \xi\in S^{k-1}$, $\Var(W\cdot \xi)=1$.
Therefore, if $\theta_1$ and $\theta_2$ are of order $1$, then typically $\alpha$ is of order $1$ and $\beta$ is of order $k$, in which case the bound in \eq{c3.3-1} is of order
$k^{1/2}(\gamma+\epsilon)$
\end{rem}
\begin{proof}
We apply Proposition \ref{p3.1} with the exchangeable pair $(W,W')$ constructed as above Corollary \ref{c3.3} and
\be{
D=D^I=W'-W=\sum_{i\in A_I\backslash \{I\}}(X_{iI}'-X_i)+X_I^*-X_I.
}
From the construction of $W'$ and $\E(X-X')^2=2\Var(X)$ where $X'$ is an independent copy of $X$, we have for $\xi \in S^{k-1}$,
\ben{\label{c3.3-4}
\E(D^i \cdot \xi)^2=2\E[\Var(Y_i\cdot \xi|X_{B_i\backslash A_i})].
}
In particular, $\E (D_j^i)^2=2\E[\Var(Y_{ij}|X_{B_i\backslash A_i})]$.
By the inequality $(a_1+\dots+a_m)^3\leq m^2(|a_1|^3+\dots+|a_m|^3)$, the bound $\theta_1$ in \eq{c3.3--1} and the definition of $\gamma$ in \eq{c3.3-0},
\besn{\label{c3.3-3}
\E |D|^3&=\frac{1}{n} \sum_{i=1}^n \E |D^i|^3 =\frac{1}{n}\sum_{i=1}^n \E |\sum_{i_1\in A_i\backslash 
\{i\}}(X_{i_1}'-X_{i_1})+X_i^*-X_i|^3\\ 
&\leq \frac{4\theta_1^2}{n} \sum_{i=1}^n\Big[ \sum_{i_1\in A_i\backslash \{i\}} (\E|X_{i_1}'|^3+\E|X_{i_1}|^3) +\E|X_i^*|^3+\E|X_i|^3 \Big]\leq \frac{8\theta_1^3}{n}\gamma.
}
By the (LD3) assumption, the inequality $\cov(X,Y)\leq (\E X^2 + \E Y^2)/2$, and the bound $\theta_2$ in \eq{c3.3--1},
\bes{
\sum_{j=1}^k \Var(\E(D_j|W))&\leq \sum_{j=1}^k \Var(\frac{1}{n}\sum_{i=1}^n D^i_j)
= \frac{1}{n^2}\sum_{j=1}^k \sum_{i=1}^n \sum_{i_1: B_{i_1}\cap C_i\ne \emptyset}\cov(D^i_j, D^{i_1}_j) \\
&\leq \frac{1}{n^2}\sum_{j=1}^k \sum_{i=1}^n \sum_{i_1: B_{i_1}\cap C_i\ne \emptyset}\frac{\E(D^i_j)^2+\E(D^{i_1}_j)^2}{2}
 \leq  \frac{\theta_2}{n^2}\beta.
}
Similarly,
\bes{
&\sum_{j,j_1=1}^k \Var\left\{ \E \left[D_j D_{j_1}I(|D|\leq \delta)|W \right] \right\} \leq \sum_{j,j_1=1}^k \Var\left[ \frac{1}{n}\sum_{i=1}^n  D^i_j D^i_{j_1}I(|D^i|\leq \delta)  \right]\\
&=\frac{1}{n^2}\sum_{j,j_1=1}^k \sum_{i=1}^n \sum_{i_1: B_{i_1}\cap C_i\ne \emptyset}
\cov\left[ D^i_j D_{j_1}^i I(|D^i|\leq \delta), D_j^{i_1} D_{j_1}^{i_1} I(|D^{i_1}|\leq \delta)  \right]\\
&\leq \sum_{j,j_1=1}^k \frac{\theta_2}{n^2}\sum_{i=1}^n \E (D^i_j)^2(D^i_{j_1})^2I(|D^i|\leq \delta)\\
&\leq \frac{\theta_2}{n^2} \sum_{i=1}^n \E|D^i|^3 \delta\leq 8\theta_1^3\theta_2 \delta \gamma/n^2,
}
where in the last inequality we used \eq{c3.3-3}.
For $\delta$ in \eq{p3.1-0}, by \eq{c3.3-3} and \eq{c3.3-4},
\be{
\delta\leq \frac{16\theta_1^3 \gamma}{\alpha}.
}
The bound \eq{c3.3-1} follows from \eq{p3.1-1} and the above bounds.
\end{proof}
\begin{rem}\label{r2.10}
In Section 4, we will obtain multivariate normal approximation results for $W=\sum_{i=1}^n X_i$ under (LD3). Define
\be{
\alpha_1:=2\inf_{\xi\in S^{k-1}} \inf_{i_1\in [n]}  \sum_{i: B_{i}\cap C_{i_1}=\emptyset} \E [\Var(Y_i\cdot \xi|X_{B_i\backslash A_i})],
}
$\beta, \gamma$ as in \eq{c3.3-0}. We will need the following conditional concentration inequalities. Letting $\P^{X_{B_i}}$ denote the conditional probability given $X_{B_i}:=\{X_{i_1}: i_1\in B_i\}$, we have
\ben{\label{r2.10-2}
\P^{X_{B_i}} (W\in A^{\epsilon+\delta_1}\backslash A^{\delta_1})\leq \frac{\theta_1^3 \sqrt{\theta_2}}{\alpha_1^{3/2}}(16\sqrt{2}+\frac{512}{3}\sqrt{\frac{\beta}{\alpha_1}})\gamma+\frac{16\sqrt{\theta_2 \beta}}{3\alpha_1}\epsilon
}
where $\theta_1, \theta_2$ are as defined in \eq{c3.3--1},  $\epsilon, A$ may depend on $X_{B_i}$, and $\delta_1=16\theta_1^3 \gamma/\alpha_1$. To prove \eq{r2.10-2}, we first regard the conditional random vector as again a sum of locally dependent random vectors, then we construct an exchangeable pair as above Corollary \ref{c3.3} but with $I$ uniformly distributed over $\{i_1: B_{i_1}\cap C_i=\emptyset\}$. Finally we apply Proposition \ref{p3.1} with the modified exchangeable pair and proceed as in the proof of Corollary \ref{c3.3}.
\end{rem}
Corollary \ref{c3.3} has the following corollary for independent case.
\begin{cor}\label{p3}
Let $k$-dimensional random vector $W$ be
\[
W=(W_1,\ldots, W_k)^t=\sum_{i=1}^n X_i=\sum_{i=1}^n (X_{i1},X_{i2},\ldots,X_{ik})^t
\]
where $\mathbb{X}:=\{X_i: i\in [n]\}$ are independent random vectors such that $\E X_i=0$ and $\E WW^t= I_{k\times k}$. Then, for any convex set $A$ in $\mathbb{R}^k$,
\besn{\label{p3-1}
\P (W^{(i)}\in A^{4\gamma+\epsilon}\backslash A^{4\gamma})\le 4.1k^{1/2}\epsilon+39k^{1/2}\gamma
}
and
\besn{\label{p3-2}
\P (W\in A^{4\gamma+|X_i|}\backslash A^{4\gamma})\le 4.1k^{1/2}\E |X_i|+39k^{1/2}\gamma
}
for any $\epsilon>0$ and $i\in [n]$ where $W^{(i)}=W-X_i$ and $\gamma=\sum_{i=1}^n \E |X_i|^3$.
\end{cor}
\begin{proof}
For independent case, with $\E W W^t=I_{k\times k}$ in Corollary \ref{c3.3}
\be{
\theta_1=\theta_2=1,\  \alpha_1=2,\  \beta=2k,\  \delta_1=8\gamma.
}
Applying \eq{c3.3-1}, we have
\be{
\P(W\in A^{\epsilon+8\gamma}\backslash A^{8\gamma})\leq  k^{1/2}\left[ \frac{16}{3\sqrt{2}}\epsilon+\left( 8+\frac{256}{3\sqrt{2}} \right)\gamma \right].
}
The concentration inequalities \eq{p3-1} and \eq{p3-2} are obtained by considering $W^{(i)}$ instead, and the constants $4, 4.1$ and $39$ are obtained by simplifying the proof of Proposition \ref{p3.1} for independent case. Details can be found in \cite{Fa12}.
\end{proof}

\section{Multivariate normal approximation for sums of independent random vectors}

In this section, we prove a multivariate normal approximation result for sums of independent random vectors by applying the concentration inequality approach in Stein's method.
A multivariate version of the Stein equation was given in \cite{Go91} as well as in \cite{Ba90} as follows.
\ben{\label{Stein equation}
\triangle f(w)- w\cdot \nabla f(w) = h(w) -\E h(Z)
}
where $h$ is a test function and $Z$ is a standard $k$-dimensional Gaussian random vector, and $\triangle$ and $\nabla$ denote Laplace and gradient operator respectively.

If the test function $h$ is smooth enough, the above equation can be solved and one of its solution can be expressed as $f(w)=\int_0^1 g(w,s)ds$ where
\ben{\label{g epsilon}
g(w,s)=-\frac{1}{2(1-s)} \int_{\mathbb{R}^k}
    [h(\sqrt{1-s}w+\sqrt{s}z)-\E h (Z)]\phi(z)dz,
}
where $\phi(z)$ is the density function of the $k$-dimensional standard normal distribution at $z\in \mathbb{R}^k$. 
We will write $\partial_j g(w,s)=\partial g(w,s)/\partial w_j$ and so on.
When $\nabla h$ is Lipschiz,
\besn{\label{Dg}
\partial_{jj_1}g(w,s)&=-\frac{1}{2s} \int_{\mathbb{R}^k} h(\sqrt{1-s}w+\sqrt{s}z) \partial_{jj_1}\phi(z) dz  \\
&\quad+  \frac{1}{2\sqrt{s}} \int_{\mathbb{R}^k} \partial_{j_1} h(\sqrt{1-s}w+\sqrt{s}z) \partial_j \phi(z) dz.
}
The class of test functions we are interested in is
$h=I_A$ where $A$ is a convex set in $\mathbb{R}^k$. A smoothed version of such an $h$ was introduced by \cite{Be03} as
\ben{\label{h epsilon}
h_\epsilon (w) = \psi(\frac{d(w, A)}{\epsilon})
}
where $\epsilon>0$ and
\beqn
\psi(x)=
\begin{cases}
1, & x<0\\
1-2x^2, & 0\le x<\frac{1}{2}\\
2(1-x)^2, & \frac{1}{2} \le x <1\\
0, & 1\le x.
\end{cases}
\eeqn
The next lemma was proved in \cite{Be03}.
\begin{lem}
The above defined function $h_\epsilon$ satisfies:
\ben{\label{h epsilon-1}
h_\epsilon(w)=1\  \text{for} \ w\in A,\quad h_\epsilon(w)=0 \ \text{for}\  w\in \mathbb{R}^k\backslash A^\epsilon, \quad 0\le h_\epsilon \le 1,
}
and
\ben{\label{h epsilon-2}
|\nabla h_\epsilon (w)|\le \frac{2}{\epsilon},\quad  |\nabla h_\epsilon (w_1)-\nabla h_\epsilon (w_2)|\le \frac{8|w_1-w_2|}{\epsilon^2}.
}
\end{lem}
For a convex set $A$ and $\gamma\ge 0$, defining $g_{1,\epsilon}=h_\epsilon$ for $h=I_{A^{4\gamma}}$, we have
\bes{
\P(W\in A)-\P(Z\in A) &\le \P(W\in A^{4\gamma})-\P(Z\in A)\\
&\le \E g_{1,\epsilon} (W) -\E g_{1,\epsilon} (Z) +\E g_{1,\epsilon}(Z)-\P(Z\in A)\\
&\le \E g_{1,\epsilon} (W)- \E g_{1,\epsilon} (Z)+\P(Z\in A^{4\gamma+\epsilon}\backslash A)\\
&\le  \E g_{1,\epsilon} (W)- \E g_{1,\epsilon} (Z)+ k^{1/2} (4\gamma+\epsilon)
}
where we used \eq{h epsilon-1} and (\ref{p1-1}). If $A^{-\epsilon-4\gamma}=\emptyset$, by \eq{p1-1},
\be{
\P(W\in A)-\P(Z\in A) \ge -\P(Z\in A\backslash A^{-\epsilon-4\gamma})\ge -k^{1/2}(4\gamma+\epsilon).
}
If not, defining $g_{2,\epsilon}=h_\epsilon$ for $h=I_{(A^{-\epsilon-4\gamma})^{4\gamma}}$, we have
again by \eq{h epsilon-1} and \eq{p1-1},
\bes{
\P(W\in A)-\P(Z\in A) &\ge  \E g_{2,\epsilon} (W) -\E g_{2,\epsilon} (Z) +\E g_{2,\epsilon}(Z)-\P(Z\in A)\\
&\ge \E g_{2,\epsilon} (W) -\E g_{2,\epsilon} (Z)-\P(Z\in A\backslash A^{-\epsilon-4\gamma})\\
&\ge \E g_{2,\epsilon} (W) -\E g_{2,\epsilon} (Z)-k^{1/2}(4\gamma+\epsilon).
}
Therefore, we have the following smoothing lemma.
\begin{lem}\label{smooth lem}
For any $k$-dimensional random vector $W$,
\ben{\label{smooth ineq-2}
\sup_{A\in \mathcal{A}}|\P (W\in A)-\P (Z\in A)| \le \sup_{h=I_{A^{4\gamma}}: A\in\mathcal{A}}|\E h_\epsilon(W)-\E h_\epsilon(Z)|+k^{1/2}(\epsilon+4\gamma)
}
where $Z$ is a standard $k$-dimensional Gaussian random vector, $\mathcal{A}$ is the set of all the convex sets in $\mathbb{R}^k$, $\epsilon>0$, $\gamma\ge 0$ and $h_\epsilon$ is defined as in (\ref{h epsilon}).
\end{lem}

The following lemma from \cite{Be03} will be used in this section.
\begin{lem}\label{lem6}
For a $k$-dimensional vector $x$,
\ben{\label{l3.3-1}
\int_{\mathbb{R}^k} |\sum_{j=1}^k x_j \partial_j \phi(z)| dz\le \sqrt{\frac{2}{\pi}}|x|,
}
\ben{\label{l3.3-2}
\int_{\mathbb{R}^k} |\sum_{j,j_1,j_2=1}^k x_j x_{j_1} x_{j_2} \partial_{jj_1j_2} \phi(z)|dz \le 2\frac{1+4 e^{-3/2}}{\sqrt{2\pi}}|x|^3.
}
\end{lem}

Using the same argument as in \cite{Be03} when proving Lemma \ref{lem6}, we obtain the following lemma.
\begin{lem}\label{l3.4}
For $k$-dimensional vectors $u,v$, we have
\ben{\label{l3.4-1}
\int_{\mathbb{R}^k} |\sum_{j,j_1,j_2=1}^k u_j v_{j_1} v_{j_2} \partial_{jj_1j_2} \phi(z)| dz \le 2(1+\sqrt{\frac{2}{\pi}})|u||v|^2.
}
\end{lem}
\begin{proof}
It is straightforward to verify that
\besn{\label{l3.4-2}
&\sum_{j,j_1,j_2=1}^k u_j v_{j_1} v_{j_2} \partial_{jj_1j_2} \phi(z) \\
&= (|v|^2 (u\cdot z)+2 (u\cdot v)(v\cdot z)-(u\cdot z)(v\cdot z)^2) \phi(z).
}
From (\ref{l3.4-2}), we only need to consider the projection of $z$ in the two-dimensional space spanned by vectors $u,v$. Therefore, the constant obtained is dimension free and the upper bound \eq{l3.4-1} can be calculated as follows. Let $Z_1, Z_2$ be two independent $1$-dimensional standard Gaussian variables, then
\bes{
&\int_{\mathbb{R}^k} |\sum_{j,j_1,j_2=1}^k u_j v_{j_1} v_{j_2} \partial_{jj_1j_2} \phi(z)| dz\\
&\le |u||v|^2 (\E |3Z_1-Z_1^3|+ \E| Z_2 (1-Z_1^2)|)\le 2(1+\sqrt{\frac{2}{\pi}})|u||v|^2.
}
\end{proof}

\begin{thm}\label{t2}
Let $k$-dimensional random vector $W$ be
\[
W=(W_1,\ldots, W_k)^t=\sum_{i=1}^n X_i=\sum_{i=1}^n (X_{i1},X_{i2},\ldots,X_{ik})^t
\]
where $\{X_i: i\in [n]\}$ are independent such that $\E X_i=0$ for each $i$ and $\E W W^t=I_{k\times k}$. Then,
\ben{\label{t2-1}
\sup_{A\in \mathcal{A}}|\P (W\in A)-\P (Z\in A)|\le 115 k^{1/2} \gamma
}
where $\mathcal{A}$ is the set of all the convex sets in $\mathbb{R}^k$, $Z$ is a standard $k$-dimensional Gaussian vector and $\gamma=\sum_{i=1}^n \gamma_i=\sum_{i=1}^n \E |X_i|^3$.
\end{thm}
\begin{proof} Without loss of generality, assume $\gamma$ is finite.
For a convex set $A\in \mathcal{A}$ and $\epsilon>0$, we define $h_{A,\epsilon}(w)=
\psi(d(w,A^{4\gamma})/\epsilon)$ as in \eq{h epsilon}. Let $g_{A,\epsilon}$ be defined as in \eq{g epsilon}
with $h$ replaced by $h_{A, \epsilon}$. As argued at the beginning of this section, $f_{A,\epsilon}(w)=\int_0^1 g_{A,\epsilon}(w,s)ds$ solves the Stein equation
\ben{\label{Stein equation 2}
\triangle f_{A,\epsilon}(w)- w\cdot \nabla f_{A,\epsilon}(w) = h_{A,\epsilon}(w) -\E h_{A,\epsilon}(Z).
}
In what follows, we keep the dependence on $A$ implicit and write $h_\epsilon=h_{A,\epsilon}$, $g_\epsilon=g_{A,\epsilon}$.

With $W^{(i)}=W-X_i$, we have by the independence assumption and $\E X_i=0$,
\bes{
& \E \triangle g_\epsilon (W,s)-\E  W\cdot \nabla g_\epsilon (W,s) \\
&=\E \triangle g_\epsilon(W,s)-\sum_{i=1}^n \E X_i \cdot (\nabla g_\epsilon(W,s)-\nabla g_\epsilon (W^{(i)},s)) \\
&=\E \triangle g_\epsilon (W,s)-\sum_{i=1}^n \E X_i\cdot (\text{Hess} g_\epsilon(W^{(i)},s)X_i) \\
&\quad -\sum_{i=1}^n \E X_i\cdot(\nabla g_\epsilon (W,s)-\nabla g_\epsilon (W^{(i)},s)-\text{Hess} g_\epsilon (W^{(i)},s)X_i) \\
&=R_1(s)-R_2(s)
}
where
\be{
R_1(s)=\sum_{i=1}^n \sum_{j,j_1=1}^k\E X_{ij}X_{ij_1}\E [\partial_{jj_1}g_\epsilon (W,s)-\partial_{jj_1}g_\epsilon(W^{(i)},s)] 
}
and
\be{
R_2(s)=\sum_{i=1}^n \sum_{j,j_1=1}^k \E X_{ij}X_{ij_1}[\partial_{jj_1}g_\epsilon (W^{(i)}+UX_i,s)-\partial_{jj_1}g_\epsilon(W^{(i)},s)] 
}
where $U$ is an independent uniform random variable in $[0,1]$.
By \eq{Stein equation 2}, 
\be{
\E h_\epsilon(W) -\E h_\epsilon(Z)=\int_0^1 (R_1(s)-R_2(s))ds.
}
For $R_2(s)$, we consider the cases $0<s\leq \epsilon^2$ and $\epsilon^2<s\leq 1$ separately. For the case $0<s\leq \epsilon^2$, we use the second expression of $\partial_{jj_1}g_\epsilon$ in \eq{Dg}, and write
\bes{
R_2(s)&=\sum_{i=1}^n  \sum_{j,j_1=1}^k \E X_{ij}X_{ij_1}\frac{1}{2\sqrt{s}}\int_{\mathbb{R}^k}\bigl[\partial_{j_1} h_\epsilon(\sqrt{1-s}W^{(i)}+\sqrt{1-s}UX_i+\sqrt{s}z)\nn \\
&\kern9em -\partial_{j_1} h_\epsilon(\sqrt{1-s}W^{(i)}+\sqrt{s}z)\bigr]\partial_{j}\phi(z)dz.
}
Introducing another independent uniform random variable $U'$ in $[0,1]$, we have
\bes{
&R_2(s)=\sum_{i=1}^n \sum_{j,j_1,j_2=1}^k \E UX_{ij}X_{ij_1}X_{ij_2}\frac{\sqrt{1-s}}{2\sqrt{s}} \\
&\kern4em \times\int_{\mathbb{R}^k} \partial_{j_1j_2} h_\epsilon(\sqrt{1-s}W^{(i)}+\sqrt{1-s}UU'X_i+\sqrt{s}z)\partial_{j}\phi(z)dz \\
&=\sum_{i=1}^n \sum_{j=1}^k \E UX_{ij} \frac{\sqrt{1-s}}{2\sqrt{s}} \\
&\quad \times \int_{\mathbb{R}^k} ( \sum_{j_1=1}^k X_{ij_1}\partial_{j_1} \nabla h_\epsilon (\sqrt{1-s}W^{(1)}+\sqrt{1-s} UU'X_i+\sqrt{s}z) \cdot X_i) \partial_j \phi(z) dz . 
}
Define any linear transform of a set to be the image of the linear transform of all the elements in the set. By (\ref{h epsilon-2}) and \eq{p3-1},
\bes{
&|\E ^{U,U',X_i}(\sum_{j_1=1}^k X_{ij_1} \partial_{j_1}\nabla h_\epsilon (\sqrt{1-s}W^{(i)}+\sqrt{s}z+\sqrt{1-s}UU'X_i)\cdot X_i)|\nn \\
&\le \frac{8}{\epsilon^2} |X_i|^2 \E ^{U,U'X_i}I(\sqrt{1-s}W^{(i)}\in A^{\epsilon+4\gamma}\backslash A^{4\gamma}-(\sqrt{s}z+\sqrt{1-s}UU'X_i))\nn \\
&\le  |X_i|^2 (32.8k^{1/2}\frac{1}{\epsilon\sqrt{1-s}}+312k^{1/2}\frac{\gamma}{\epsilon^2}).\nn
}
Therefore,
\besn{\label{R22}
|\int_0^{\epsilon^2} R_2(s)ds|&\le \frac{1}{2} \sum_{i=1}^n  \E  |X_i|^2 \int_0^{\epsilon^2} \frac{\sqrt{1-s}}{2\sqrt{s}}(32.8k^{1/2}\frac{1}{\epsilon\sqrt{1-s}}+312k^{1/2}\frac{\gamma}{\epsilon^2}) \\
&\times\int_{\mathbb{R}^k}|\sum_{j=1}^k X_{ij} \partial_j\phi(z)   |dzds \\
&\le \sqrt{\frac{2}{\pi}}\gamma(16.4k^{1/2}+156k^{1/2}\frac{\gamma}{\epsilon})
}
where we used \eq{l3.3-1}.
For the case $\epsilon^2<s\leq 1$, using the first expression of $\partial_{jj_1}$ in \eq{Dg} and the integration by parts formula,
\bes{
R_2(s)&=\sum_{i=1}^n \sum_{j,j_1=1}^k \E X_{ij}X_{ij_1}(-\frac{1}{2s})\int_{\mathbb{R}^k}\bigl[h_\epsilon(\sqrt{1-s}W^{(i)}+\sqrt{1-s}UX_i+\sqrt{s}z)\nn \\
&\kern9em -h_\epsilon(\sqrt{1-s}W^{(i)}+\sqrt{s}z)\bigr]\partial_{jj_1}\phi(z)dz\\
&=\sum_{i=1}^n \sum_{j,j_1,j_2=1}^k \E UX_{ij}X_{ij_1}X_{ij_2} \frac{\sqrt{1-s}}{2s^{3/2}}\\
&\kern9em\times \int_{\mathbb{R}^k} [h_\epsilon(\sqrt{1-s}W^{(i)}+\sqrt{s}z+\sqrt{1-s}UU'X_i)\phi(z)dz.
}
Write $R_2(s)=R_2'(s)+R_2''(s)$ by separating the sum over $i$ into two parts according to $\gamma_i \le 8\gamma^3$ or else. Write $R_2'(s)=R_{2,1}'(s)+R_{2,2}'(s)$ by subtracting a term with $W^{(i)}$ replaced by an independent $k$-dimensional standard Gaussian vector $Z$ and adding the same term, i.e.,
\bes{
R_{2,1}'(s)&=\sum_{i: \gamma_i\le 8\gamma^3} \sum_{j,j_1,j_2=1}^k \E UX_{ij}X_{ij_1}X_{ij_2} \frac{\sqrt{1-s}}{2s^{3/2}} \\
&\quad \times \int_{\mathbb{R}^k} [h_\epsilon(\sqrt{1-s}W^{(i)}+\sqrt{s}z+\sqrt{1-s}UU'X_i)\nn \\
&\kern4em \quad -h_\epsilon(\sqrt{1-s}Z+\sqrt{s}z+\sqrt{1-s}UU'X_i)]\partial_{jj_1j_2}\phi(z)dz
}
and
\bes{
R_{2,2}'(s)&=\sum_{i: \gamma_i\le 8\gamma^3} \sum_{j,j_1,j_2=1}^k \E UX_{ij}X_{ij_1}X_{ij_1'} \frac{\sqrt{1-s}}{2s^{3/2}}\nn \\
&\quad \times \int_{\mathbb{R}^k} 
h_\epsilon(\sqrt{1-s}Z+\sqrt{s}z+\sqrt{1-s}UU'X_i)\partial_{jj_1j_2}\phi(z)dz.\nn
}
By introducing an independent copy $\widetilde{X}_i$ of $X_i$, $\widetilde{W}=W^{(i)}+\widetilde{X}_i$ has the same distribution as $W$ and is independent of $X_i$. We have
\bes{
&\E ^{U,U',X_i} \bigl\{ h_\epsilon(\sqrt{1-s}W^{(i)}+\sqrt{s}z+\sqrt{1-s}UU'X_i)\\
&\kern4em -h_\epsilon(\sqrt{1-s}Z+\sqrt{s}z+\sqrt{1-s}UU'X_i) \bigr\} \\
&\le \E ^{U,U',X_i} \Bigl\{ I(W^{(i)} \in \frac{1}{\sqrt{1-s}}(A^{4\gamma+\epsilon}-\sqrt{s}z-\sqrt{1-s}UU'X_i))\\
&\kern5em -  I(Z \in \frac{1}{\sqrt{1-s}}(A^{4\gamma}-\sqrt{s}z-\sqrt{1-s}UU'X_i)) \Bigr\}\\
&\le \E ^{U,U',X_i} \biggl\{ I \bigl[W^{(i)}+\widetilde{X}_i\in \bigl( \frac{1}{\sqrt{1-s}}(A^{4\gamma+\epsilon}-\sqrt{s}z-\sqrt{1-s}UU'X_i) \bigr)^{|\widetilde{X}_i|}\\
&\kern12em  \backslash \frac{1}{\sqrt{1-s}}(A^{4\gamma+\epsilon}-\sqrt{s}z-\sqrt{1-s}UU'X_i)   \bigr]\\
&\kern6em + I(Z\in \frac{1}{\sqrt{1-s}}(A^{4\gamma+\epsilon}-\sqrt{s}z-\sqrt{1-s}UU'X_i) \\
&\kern13em  \backslash \frac{1}{\sqrt{1-s}}(A^{4\gamma}-\sqrt{s}z-\sqrt{1-s}UU'X_i))\\
&\kern6em +I(\widetilde{W} \in \frac{1}{\sqrt{1-s}}(A^{4\gamma+\epsilon}-\sqrt{s}z-\sqrt{1-s}UU'X_i)) \\
&\kern6em -I(Z\in \frac{1}{\sqrt{1-s}}(A^{4\gamma+\epsilon}-\sqrt{s}z-\sqrt{1-s}UU'X_i)) \biggr\}.
}
Let $\delta_\gamma$ denote the supreme of $\sup_{A\in \mathcal{A}}|\P (W\in A)-\P (Z\in A)|$ over all $W$ such that $W$ can be expressed as sum of $n$ independent mean $0$ random vectors such that $\cov(W, W)=I_{k\times k}$ and the sum of absolute third moments of the summands is bounded by $\gamma$. Using the concentration inequalities \eq{p3-2} and \eq{p1-1} and the definition of $\delta_\gamma$, we have
\besn{\label{applicationofsecondconcentrationinequality}
&\E ^{U,U',X_i}[h_\epsilon(\sqrt{1-s}W^{(i)}+\sqrt{s}z+\sqrt{1-s}UU'X_i) \\
&\quad-h_\epsilon(\sqrt{1-s}Z+\sqrt{s}z+\sqrt{1-s}UU'X_i)] \\
&\le  4.1 k^{1/2} \E |\widetilde{X}_i|+39 k^{1/2}\gamma +k^{1/2}\frac{\epsilon}{\sqrt{1-s}}+\delta_\gamma.
}
After proving a lower bound in same way as proving the upper bound \eq{applicationofsecondconcentrationinequality}, we have, by \eq{l3.3-2},
\besn{\label{R21'}
&|\int_{\epsilon^2}^1 R_{2,1}'(s)ds|\\
&\leq  \frac{1+4e^{-3/2}}{\sqrt{2\pi}} \sum_{i: \gamma_i\le 8\gamma^3} \frac{1}{\epsilon} 
\big[  4.1 k^{1/2} (\E |\widetilde{X}_i|^3)^{1/2}+39 k^{1/2}\gamma +k^{1/2}\epsilon+\delta_\gamma\big] \E|X_i|^3 \\ 
&\leq \frac{1+4e^{-3/2}}{\sqrt{2\pi}}(47.2k^{1/2}\frac{\gamma}{\epsilon}+k^{1/2}+\frac{\delta_\gamma}{\epsilon}) \sum_{i: \gamma_i\le 8\gamma^3} \E |X_i|^3.
}
For $R_{2,2}'(s)$, using the integration by parts formula and observing that $\sqrt{1-s}Z+\sqrt{s}\widetilde{Z}$ has the same distribution as $Z$ where $\widetilde{Z}$ is an independent copy of standard normal $Z$,
\bes{
&\E^{X_i} \int_{\epsilon^2}^1 \frac{\sqrt{1-s}}{2s^{3/2}}\int_{\mathbb{R}^k} h_\epsilon(\sqrt{1-s}Z+\sqrt{s}z+\sqrt{1-s}UU'X_i)\partial_{jj_1j_2}\phi(z)dzds\nn \\
&= -\E^{X_i} \int_{\epsilon^2}^1 \frac{\sqrt{1-s}}{2}\int_{\mathbb{R}^k}\partial_{jj_1j_2} h_\epsilon(\sqrt{1-s}Z+\sqrt{s}z+\sqrt{1-s}UU'X_i)\phi(z)dzds\nn \\
&=\int_{\epsilon^2}^1 \frac{\sqrt{1-s}}{2} \int_{\mathbb{R}^k} h_\epsilon(z+\sqrt{1-s}UU'X_i) \partial_{jj_1j_2}\phi(z)dzds.
}
Therefore, by \eq{l3.3-2}, $\E U=1/2$ and $\int_0^1 \sqrt{1-s}ds=2/3$,
\ben{\label{R212'}
|\int_{\epsilon^2}^1 R_{2,2}'(s)|\le \frac{1+4e^{-3/2}}{3\sqrt{2\pi}}\sum_{i:\gamma_i\le 8\gamma^3} \E |X_i|^3.
}
We remark that in the above calculation we used the third derivatives of $h_\epsilon$ which does not exist. However, we can smooth $h_\epsilon$ first then use limiting arguments to show that the final equality holds even if $h_\epsilon$ does not have third derivatives. Now we turn to bounding $\int_{\epsilon^2}^1 R_2''(s)ds$ where
\bes{
R_2''(s)&=\sum_{i: \gamma_i>8\gamma^3} \sum_{j,j_1,j_2=1}^k \E UX_{ij}X_{ij_1}X_{ij_2}\frac{\sqrt{1-s}}{2s^{3/2}}\nn \\
&\quad \times \int_{\mathbb{R}^k} h_\epsilon(\sqrt{1-s}W^{(i)}+\sqrt{1-s}UU'X_i+\sqrt{s}z)\partial_{jj_1j_2}\phi(z)dz.
}
For each $X_i$ such that $\gamma_i >8\gamma^3$, define $N_i$ to be the positive square root of the inverse of the matrix $I_{k\times k}-\cov (X_i, X_i)$. Then we have the following bound on the operator norm of $N_i$.
\besn{\label{Ni}
||N_i||&=\sqrt{||(I_{k\times k}-\cov (X_i, X_i))^{-1}||}\le (\frac{1}{1-||\cov (X_i, X_i)||})^{1/2} \\
&= (\frac{1}{1-\sup_{|u|=1} u'\cov(X_i, X_i) u})^{1/2} = (\frac{1}{1-\sup_{|u|=1}E(u'X_i)^2})^{1/2} \\
&\le (\frac{1}{1-E|X_i|^2})^{1/2} \le (\frac{1}{1-\gamma_i^{2/3}})^{1/2}.
}
Note that
\bes{
N_iW^{(i)}=\sum_{i_1: i_1\ne i} N_i X_{i_1}
}
is a sum of $n$ independent random vectors (with one $0$-vector) with
\bes{
\E N_i X_{i_1}=0, \quad \cov (N_iW^{(i)}, N_i W^{(i)})=I_{k\times k}
} 
and
\bes{
\sum_{i_1: i_1\ne i} \E |N_i X_{i_1}|^3&\le \frac{\gamma-\gamma_i}{(1-\gamma_i^{2/3})^{3/2}}
\le  \frac{\gamma-\gamma_i}{(1-\gamma_i^{2/3})^{2}}\le \frac{\gamma-\gamma_i}{1-2\gamma_i ^{2/3}}\le \gamma
}
where we used the fact that $\gamma_i >8\gamma^3$ in the last inequality. Therefore, $N_iW^{(i)}$ can be regarded as a standardized sum of $n$ independent random vectors  with sum of absolute third moments of the summands less than $\gamma$. We write $R_2''(s)$ into two parts as
\bes{
R_{2,1}''(s)&=\sum_{i: \gamma_i>8\gamma^3}\sum_{j,j_1,j_2=1}^k \E UX_{ij}X_{ij_1}X_{ij_2} \frac{\sqrt{1-s}}{2s^{3/2}} \\
&\qquad \times \int_{\mathbb{R}^k} [h_\epsilon(\sqrt{1-s}N_i^{-1} (N_iW^{(i)})+\sqrt{s}z+\sqrt{1-s}UU'X_i) \\
&\kern5em -h_\epsilon(\sqrt{1-s}N_i^{-1} Z+\sqrt{s}z+\sqrt{1-s}UU'X_i)]\partial_{jj_1j_2}\phi(z)dz
}
and
\bes{
R_{2,2}''(s)&=\sum_{i: \gamma_i>8\gamma^3}\sum_{j,j_1,j_2=1}^k \E UX_{ij}X_{ij_1}X_{ij_2} \frac{\sqrt{1-s}}{2s^{3/2}}\nn \\
&\quad \times \int_{\mathbb{R}^k} 
h_\epsilon(\sqrt{1-s}N_i^{-1}Z+\sqrt{s}z+\sqrt{1-s}UU'X_i)\partial_{jj_1j_2}\phi(z)dz.
}
By the definition of $\delta_\gamma$ above \eq{applicationofsecondconcentrationinequality} and \eq{p1-1},
\bes{
&\E^{U,U',X_i} [h_\epsilon(\sqrt{1-s}N_i^{-1} (N_iW^{(i)})+\sqrt{s}z+\sqrt{1-s}UU'X_i) \\
&\kern4em -h_\epsilon(\sqrt{1-s}N_i^{-1} Z+\sqrt{s}z+\sqrt{1-s}UU'X_i)]\\
&\le \E^{U,U',X_i} [I(N_iW^{(i)} \in \frac{N_i}{\sqrt{1-s}}
(A^{4\gamma+\epsilon}-\sqrt{s}z-\sqrt{1-s}UU'X_i))\\
&\kern6em -I(Z\in \frac{N_i}{\sqrt{1-s}}
(A^{4\gamma+\epsilon}-\sqrt{s}z-\sqrt{1-s}UU'X_i))\\
&\kern6em +I(Z\in \frac{N_i}{\sqrt{1-s}}
(A^{4\gamma+\epsilon}-\sqrt{s}z-\sqrt{1-s}UU'X_i))\\
&\kern6em -I(Z \in \frac{N_i}{\sqrt{1-s}}(A^{4\gamma}-\sqrt{s}z-\sqrt{1-s}UU'X_i))]\\
&\le \delta_\gamma+k^{1/2}\frac{\epsilon}{\sqrt{1-s}} ||N_i||
}
Along with a similar lower bound, we have by \eq{l3.3-2}
\ben{\label{R21''}
|\int_{\epsilon^2}^1R_{2,1}''(s)ds| \le \sum_{i:\gamma_i>8\gamma^3} \frac{1+4e^{-3/2}}{\sqrt{2\pi}} \E |X_i|^3 (\frac{\delta_\gamma}{\epsilon}+k^{1/2}\frac{1}{\sqrt{1-\gamma^{2/3}}}).
}
Using a similar argument leading to (\ref{R212'}), $R_{2,2}''(s)$ can be written as
\bes{
R_{2,1,2}'' &=\sum_{i:\gamma_i>8\gamma^3} \sum_{j,j_1,j_2=1}^k \E U X_{ij} X_{ij_1}X_{ij_2} \int_{\epsilon^2}^1 \frac{\sqrt{1-s}}{2}  \\
&\quad \times \int_{\mathbb{R}^k}h_\epsilon (Z+\sqrt{1-s}UU'X_i)\partial_{jj_1j_2}\phi_{\Sigma_i^s}(z) dz
}
where $\Sigma_i^s =I_{k\times k}-(1-s)\cov (X_i, X_i)$ and $\phi_{\Sigma_i^s}$ is the density function of $N(0,\Sigma_i^s)$.
\ignore{
$|N|$ denotes the determinant of matrix $N$.
$$\phi_{\Sigma}(z)=|\Sigma^{-1/2}|\phi(\Sigma^{-1/2}z)$$
$$|\sum_{j,j_1,j_2=1}^k u_j u_{j_1} u_{j_2} \partial_{jj_1j_2}\phi_{\Sigma}(z)|=|N||\sum_{i,i_1,i_2=1}^k (Nu)_i (Nu)_{i_1} (Nu)_{i_2} \phi_{ii_1i_2}(Nz)|$$
}
From
\bes{
& \int_{\mathbb{R}^k} |\sum_{j,j_1,j_2=1}^k  X_{ij} X_{ij_1}X_{ij_2} \partial_{jj_1j_2}\phi_{\Sigma_i^s}(z) |dz\\
&= \int_{\mathbb{R}^k} |\sum_{j,j_1,j_2=1}^k (N_i^s X_i)_j (N_i^s X_i)_{j_1} (N_i^s X_i)_{j_2} \partial_{jj_1j_2} \phi(z)  |dz
}
where $N_i^s$ is the positive square root of the inverse of $\Sigma_i^s$,
\ben{\label{R22''}
|\int_{\epsilon^2}^1 R_{2,2}''(s)ds|\le \sum_{i:\gamma_i>8\gamma^3} \E |X_i|^3 \frac{1+4e^{-3/2}}{3\sqrt{2\pi}} (\frac{1}{1-\gamma^{2/3}})^{3/2}
}
where we used the fact that $||N_i^s||\le (\frac{1}{1-\gamma^{2/3}})^{1/2}$, which can be proved as in (\ref{Ni}).
By \eq{R21'}, \eq{R212'}, \eq{R21''} and \eq{R22''},
\besn{\label{201}
|\int_{\epsilon^2}^1 R_2(s)ds|&\leq |\int_{\epsilon^2}^1 R_{2,1}'(s)ds|+|\int_{\epsilon^2}^1 R_{2,2}'(s)ds|+|\int_{\epsilon^2}^1 R_{2,1}''(s)ds|+|\int_{\epsilon^2}^1 R_{2,2}''(s)ds|\\
&\leq \frac{1+4e^{-3/2}}{\sqrt{2\pi}} \gamma(\frac{\delta_\gamma}{\epsilon}+k^{1/2}\frac{1}{\sqrt{1-\gamma^{2/3}}}+47.2 k^{1/2}\frac{\gamma}{\epsilon})\\
&\quad + \frac{1+4e^{-3/2}}{\sqrt{2\pi}} \gamma (\frac{1}{1-\gamma^{2/3}})^{3/2}.
}
Next we bound $\int_0^1 R_1(s)ds$.
Observing that $R_1(s)$ can be written as
\bes{
R_1(s)=\sum_{i=1}^n \sum_{j,j_1=1}^k \E \widetilde{X}_{ij} \widetilde{X}_{ij_1} [\partial_{jj_1}g_\epsilon (W,s)-\partial_{jj_1}g_\epsilon(W^{(i)},s)]
}
where $\widetilde{X}_i$ is an independent copy of $X_i$, we can bound it similarly as for $R_2(s)$ as follows.
\ben{\label{R12}
|\int_0^{\epsilon^2} R_1(s)ds|\le 2\sqrt{\frac{2}{\pi}} \gamma(16.4k^{1/2}+156k^{1/2}\frac{\gamma}{\epsilon}),
}
\besn{\label{R111}
|\int_{\epsilon^2}^1 R_1(s)ds|&\le 2(1+\sqrt{\frac{2}{\pi}})\gamma (\frac{\delta_\gamma}{\epsilon}+k^{1/2}\frac{1}{\sqrt{1-\gamma^{2/3}}}+47.2 k^{1/2}\frac{\gamma}{\epsilon})\\
&\quad +\frac{2}{3}(1+\sqrt{\frac{2}{\pi}})\gamma (\frac{1}{1-\gamma^{2/3}})^{3/2}.
}
Note that the constants are different from those of $R_2$ because we use (\ref{l3.4-1}) instead of (\ref{l3.3-2}) and an extra $2$ comes from the fact that there is no $U$ in $R_1$.
From the bounds (\ref{R111}), (\ref{R12}), \eq{201}, (\ref{R22}) and the smoothing inequality (\ref{smooth ineq-2}), with $c_0=2(1+\sqrt{\frac{2}{\pi}}) +\frac{1+4e^{-3/2}}{\sqrt{2\pi}}$,
\bes{
(1-\frac{\gamma c_0}{\epsilon}) \delta_\gamma&\le (49.2\sqrt{\frac{2}{\pi}}+\frac{c_0}{\sqrt{1-\gamma^{2/3}}}+\frac{c_0}{3(1-\gamma^{2/3})^{3/2}})k^{1/2}\gamma \\
&\quad +(468\sqrt{\frac{2}{\pi}}+47.2c_0)k^{1/2} \frac{\gamma^2}{\epsilon}+k^{1/2} (4\gamma+\epsilon).
}
Let $\epsilon=33\gamma$, and without loss of generality let $\gamma\le 1/115$. The bound (\ref{t2-1}) is proved by solving the above inequality.

\end{proof}

\section{Multivariate normal approximation under local dependence}

In this section, we prove multivariate normal approximation results for sums of locally dependent 
$k$-dimensional random vectors $W=\sum_{i=1}^n X_i$. 
In the first theorem, we assume (LD5), that is, in addition to (LD3), for each $i\in [n]$, there exist neighborhoods $D_i, E_i$ such that $\{X_{i_1}: i_1\in C_i\}$ is independent of $\{X_{i_1}: i_1\notin D_i\}$ and $\{X_{i_1}: i_1\in D_i\}$ is independent of $\{X_{i_1}: i_1\notin E_i\}$. Assuming further the $X_i$ have finite fourth moments, we get a bound which is typically of order $O_k(1/\sqrt{n})$. In the second theorem, we assume (LD3) and that the $X_i$ have finite third moments, and get a bound typically of order $O_k(\log n/\sqrt{n})$. The extra term $\log(n)$ also appeared in the result of \cite{RiRo96} where they assumed that the $X_i$ are bounded uniformly.
\begin{thm}\label{t4.2}
Let $k$-dimensional random vector $W=\sum_{i=1}^n X_i$ be a sum of locally dependent random vectors such that $\E X_i=0$ for each $i$ and $\E W W^t=I_{k\times k}$.
Assume (LD5) with neighborhood size bounded by
\bes{
&|A_i|,\  \max_{i\in [n]}|\{i_1: i\in A_{i_1}\}|\leq \theta_1,\\
&\max_{i\in [n]}|\{i_1: B_{i_1}\cap C_i\ne \emptyset\}|,\   \max_{i\in [n]}|\{i_1: B_{i}\cap C_{i_1}\ne \emptyset\}|  \leq \theta_2.
}
Let
\be{
Y_i=\sum_{i_1\in A_i} X_{i_1},\ Z_i=\sum_{i_1 \in B_i\backslash A_i} X_{i_1},\ T_i=\sum_{i_1 \in C_i\backslash B_i} X_{i_1}.
}
Let $\beta, \gamma$ be defined as in \eq{c3.3-0}, and let
\ben{\label{t4.2-4}
\alpha_2=2\inf_{\xi\in S^{k-1}} \inf_{i_1\in [n]}  \sum_{i: C_{i}\cap E_{i_1}=\emptyset} \E [\Var(Y_i\cdot \xi|X_{B_i\backslash A_i})]
}
where we recall $S^{k-1}$ denotes the unit $(k-1)$-dimensional sphere and $X_{B_i\backslash A_i}=\{X_{i_1}: i_1\in B_i\backslash A_i\}$.
Then we have
\ben{\label{t4.2-1}
\sup_{A\in \mathcal{A}}|\P (W\in A)-\P (Z\in A)|\le \frac{c}{1\wedge \alpha_2^{3/2}}   \theta_1^3 \sqrt{\theta_2} \sqrt{\frac{\beta}{\alpha_2}} (\gamma +b_1+b_2^{1/2})
}
where
\ben{\label{t4.2-2}
b_1=\sum_{i=1}^n\E \big( |X_i||Y_i|+\E|X_i||Y_I|  \big) (|Y_i|+|Z_i|),
}
\besn{\label{t4.2-3}
&b_2=\sum_{i=1}^n \E|X_i||Y_i|^2 (|Y_i|+|Y_i'|+|Z_i|+|Z_i'|)\\
&+ \sum_{i=1}^n \E (|X_i||Y_i|+\E|X_i||Y_i|)(|Y_i|+|Z_i|)(|Y_i|+|Y_i''|+|Z_i|+|Z_i''|+|T_i|+|T_i''|),
}
where $Y_i', Z_i'$ is an independent sample of $Y_i, Z_i$ given $\{X_{i_1}: i_1\notin B_i\}$ and $Y_i'', Z_i'', T_i''$ is an independent sample of $Y_i, Z_i, T_i$ given $\{X_{i_1}: i_1\notin C_i\}$,
and where $c$ is an absolute constant, $\mathcal{A}$ is the set of all the convex sets in $\mathbb{R}^k$ and $Z$ is a standard $k$-dimensional Gaussian vector.
\end{thm}

\begin{thm}\label{t4.1}
Let $k$-dimensional random vector $W=\sum_{i=1}^n X_i$ be a sum of locally dependent random vectors such that $\E X_i=0$ for each $i$ and $\E W W^t=I_{k\times k}$.
Assume (LD3) with neighborhood size bounded by
\besn{\label{t4.1-1}
&|B_i|,\  \max_{i\in [n]}|\{i_1: i\in B_{i_1}\}|\leq \theta_1',\\
&\max_{i\in [n]}|\{i_1: B_{i_1}\cap C_i\ne \emptyset\}|\   \max_{i\in [n]}|\{i_1: B_i\cap C_{i_1}\ne \emptyset\}|  \leq \theta_2.
}
Let $\alpha_1, \beta, \gamma$ be defined as in Corollary \ref{c3.3} and Remark \ref{r2.10}.
Then,
\ben{\label{t4.1-2}
\sup_{A\in \mathcal{A}}|\P (W\in A)-\P (Z\in A)|\le \frac{c}{1\wedge \alpha_1^{3/2}}   \theta_1'^3 \sqrt{\theta_2} \sqrt{\frac{\beta}{\alpha_1}} \gamma |\log \gamma|
}
where $c$ is an absolute constant, $\mathcal{A}$ is the set of all the convex sets in $\mathbb{R}^k$ and $Z$ is a standard $k$-dimensional Gaussian vector.
\end{thm}

\subsection{Applications}

A typical class of examples exhibiting local dependence structures is graph dependence. We first consider a special case of graph dependence. Let $G_n$ be a regular graph, with $n$ vertices and vertex degree $m$. Therefore, $G_n$ has $N=nm/2$ edges. Label the vertices by $\{1,\dots, n\}$. Let $\{\xi_{u}, u\in [n]\}$ be i.i.d. random variables taking values in $\mathcal{X}$. Let $f: \mathcal{X}\times \mathcal{X}\rightarrow \mathbb{R}^k$ be a function, and let
\be{
W=\sum_{u\sim v}f(\xi_u, \xi_v)=:\sum_{i\in E} X_i
}
where $u\sim v$ means that there is an edge in $G_n$ connecting $u$ and $v$ and $E$ denotes the edge set of $G_n$. Assume that for $u\ne v, u,v\in [n]$,  $\E f(\xi_u, \xi_v)=0$, $\E|f(\xi_u, \xi_v)|^4\leq \infty$, and $\cov(W, W)=I_{k\times k}$.
Because $X_i$ only depends on $\xi_u$ and $\xi_v$, $W$ can be regarded as a sum of locally dependent random vectors satisfying (LD5) with neighborhoods
\bes{
&A_i=\{ i_1\in E, \{i_1\}\cap \{i\}\ne \emptyset  \},\\
&B_i=\{ i_1\in E, \{i_1\}\cap A_i\ne \emptyset  \}, \\
&C_i=\{ i_1\in E, \{i_1\}\cap B_i\ne \emptyset  \}, \\
&D_i=\{ i_1\in E, \{i_1\}\cap C_i\ne \emptyset  \},\\
&E_i=\{ i_1\in E, \{i_1\}\cap D_i\ne \emptyset  \},
}
where the intersection of subsets of edge sets of $G_n$ is defined to be the set of vertices belong to both edge sets. 
From Theorem \ref{t4.2}, we conclude that
\ben{\label{e4.1-1}
\sup_{A\in \mathcal{A}}|\P (W\in A)-\P (Z\in A)|\le \frac{c}{1\wedge \alpha_2^{3/2}}   m^{11/2} \sqrt{\frac{\beta}{\alpha_2}} (\gamma +b_1+b_2^{1/2})
}
where $\alpha_2$ is defined in \eq{t4.2-4}, $\beta$ is defined in \eq{c3.3-0},
$b_1, b_2$ are defined in \eq{t4.2-2}, \eq{t4.2-3}, $c$ is an absolute constant, $\mathcal{A}$ is the set of all the convex sets in $\mathbb{R}^k$ and $Z$ is a standard $k$-dimensional Gaussian vector.

In principle, Theorem \ref{t4.2} and \ref{t4.1} can be applied to sums of random vectors with graph dependence structure to obtain results similar to \eq{e4.1-1}. Although a general result for graph dependence can be formulated, it is too tedious to write out. Instead, we give another example below which captures the details involved in the calculation of the bound. 
We consider the joint distribution of sums of partial products in an i.i.d. sequence.
In this example, we assume the existence of only finite third moments. Therefore, we shall apply Theorem \ref{t4.1}.
\begin{exa}
Let $\{\eta_1,\dots, \eta_n\}$ be i.i.d. random variables with $\E \eta_i=0, \Var (\eta_i)=1, \E|\eta_i|^3<\infty$. Define $\eta_{n+i}=\eta_i$ for all $i\in \mathbb{Z}$. For an integer $k\geq 1$, define
\be{
X_i=(X_{i1}, \dots, X_{ik})^t,\quad W=\sum_{i=1}^n X_i 
}
where for $1\leq j\leq k$,
\be{
X_{ij}=\eta_i\eta_{i+1}\dots \eta_{i+j-1}/\sqrt{n}.
}
Then $W$ is a sum of locally dependent random vectors with neighborhoods
\bes{
&A_i=\{i-k+1,\dots, i+k-1\},\\
&B_i=\{i-2k+2,\dots, i+2k-2\},\\
&C_i=\{i-3k+3,\dots, i+3k-3\}
}
for each $i\in [n]$ where $i+kn:=i$ for $k\in \mathbb{Z}$. Therefore, we can choose $\theta_1'=4k-3, \theta_2=10k-9$ in \eq{t4.1-1}. Let $\xi\in S^{k-1}$, and let $Y_i=\sum_{i_1=i-k+1}^{i+k-1}X_{i_1}$. We have
\bes{
&\E \Var (Y_i\cdot \xi | X_{B_i\backslash A_i})\\
&=\E \Var (\xi^t Y_i | \eta_{i-k+1},\dots,\eta_{i-1},\eta_{i+k},\dots,\eta_{i+2k-2})\\
&=\E \Var (\xi^t (Y_i-R_i) | \eta_{i-k+1},\dots,\eta_{i-1},\eta_{i+k},\dots,\eta_{i+2k-2})
}
where $R_i$ contains the terms in $Y_i$ which only depend on
$\{\eta_{i-k+1},\dots,\eta_{i-1}\}$.
After subtracting $R_i$, the conditional mean of $Y_i-R_i$ is $0$ and $\E (Y_i-R_i)(Y_i-R_i)^t$ is a diagonal matrix with diagonal components $(\frac{k}{n}, \dots, \frac{2k-1}{n})$. Therefore,
\bes{
&\E \Var (Y_i\cdot \xi | X_{B_i\backslash A_i})\\
&=\E (\xi^t (Y_i-R_i)(Y_i-R_i)^t \xi)\\
&=\frac{1}{n}[k\xi_1^2+(k+1)\xi_2^2+\dots+(2k-1)\xi_k^2].
}
By the above equation, $\alpha_1$ and $\beta$ in Corollary \ref{c3.3} and Remark \ref{r2.10} can be calculated as
\be{
\alpha_1=2(1-\frac{10k-9}{n})k, \quad \beta=3k^2-k.
}
Applying Theorem \ref{t4.1}, we have
\be{
\sup_{A\in \mathcal{A}}|\P (W\in A)-\P (Z\in A)|\le ck^4\frac{\log n}{\sqrt{n}} \E|\Omega|^3
}
where
\be{
\Omega=\sqrt{n}X_1=(\eta_1,\eta_1 \eta_2, \dots, \eta_1\cdots \eta_k)^t,
}
$c$ is an absolute constant, $\mathcal{A}$ is the set of all the convex sets in $\mathbb{R}^k$ and $Z$ is a standard $k$-dimensional Gaussian vector. An upper bound of $\E|\Omega|^3$ can be obtained as
\be{
\E|\Omega|^3\leq \E(|\eta_1|+|\eta_1 \eta_2|+\dots |\eta_1\dots \eta_k|)^3\leq k^3 (\E|\eta_1|^3)^k.
}
\end{exa}

\section{Proofs of Theorem \ref{t4.2} and \ref{t4.1}}
In this section, we give proofs of Theorem \ref{t4.2} and \ref{t4.1}. In the proofs, let $c$ be positive absolute constants which may differ in different expressions.

\begin{proof}[Proof of Theorem \ref{t4.2}]
Given $X_{D_i}=\{X_{i_1}: i_1\in D_i\}$, $W$ can be regarded as a sum of locally dependent random vectors. Using the same argument leading to \eq{c3.3-1} (see also Remark \ref{r2.10}), we have the following conditional concentration inequality.
\ben{\label{r2.10-3}
\P^{X_{D_i}, \mathcal{F}} (W\in A^{\epsilon+\delta_2}\backslash A^{\delta_2})\leq \frac{\theta_1^3 \sqrt{\theta_2}}{\alpha_2^{3/2}}(16\sqrt{2}+\frac{512}{3}\sqrt{\frac{\beta}{\alpha_2}})\gamma+\frac{16\sqrt{\theta_2 \beta}}{3\alpha_2}\epsilon
}
where given $X_{D_i}$, $W$ is independent of the $\sigma$-field $\mathcal{F}$,
$\epsilon, A$ may depend on $X_{D_i}$ and $\mathcal{F}$, and $\delta_2=16\theta_1^3 \gamma/\alpha_2$.

Proceed as in the proof of Theorem \ref{t2}. For a convex set $A\in \mathcal{A}$ and $\epsilon>0$, we define $h_{A,\epsilon}(w)=
\psi(d(w,A^{\delta_2})/\epsilon)$ as in \eq{h epsilon}. Let $g_{A,\epsilon}$ be defined as in \eq{g epsilon}
with $h$ replaced by $h_{A, \epsilon}$. Then $f_{A,\epsilon}(w)=\int_0^1 g_{A,\epsilon}(w,s)ds$ solves the Stein equation \eq{Stein equation 2}.
In what follows, we keep the dependence on $A$ implicit and write $h_\epsilon=h_{A,\epsilon}$, $g_\epsilon=g_{A,\epsilon}$.
We have the following smoothing inequality which is proved as for \eq{smooth ineq-2}:
\ben{\label{smooth ineq-4}
\sup_{A\in \mathcal{A}}|\P (W\in A)-\P (Z\in A)| \le \sup_{h=I_{A^{\delta_2}}: A\in\mathcal{A}}|\E h_\epsilon(W)-\E h_\epsilon(Z)|+k^{1/2}(\epsilon+\delta_2).
}
From the (LD3) assumption, with
\be{
Y_i:=\sum_{i_1\in A_i} X_{i_1},\quad V_i:=W-Y_i,
}
we have
\bes{
& \E \triangle g_\epsilon (W,s)-\E  W\cdot \nabla g_\epsilon (W,s) \\
&=\E \triangle g_\epsilon(W,s)-\sum_{i=1}^n \E X_i \cdot (\nabla g_\epsilon(W,s)-\nabla g_\epsilon (V_i,s)) \\
&=\Big[ \E \triangle g_\epsilon (W,s)-\sum_{i=1}^n \E X_i\cdot (\text{Hess} g_\epsilon(V_i,s)Y_i)\Big] \\
&\quad -\sum_{i=1}^n \E X_i\cdot(\nabla g_\epsilon (W,s)-\nabla g_\epsilon (V_i,s)-\text{Hess} g_\epsilon (V_i,s)Y_i) \\
&=:R_1(s)-R_2(s).
}
By \eq{Stein equation 2}, 
\be{
\E h_\epsilon(W) -\E h_\epsilon(Z)=\int_0^1 (R_1(s)-R_2(s))ds.
}
Let $U$ be an independent uniform random variable in $[0,1]$. We consider the cases $0<s\leq \epsilon^2$ and $\epsilon^2<s\leq 1$ separately. For the case $0<s\leq \epsilon^2$, we use the second expression of $\partial_{jj_1}g$ in \eq{Dg}, and write
\bes{
R_2(s)&=\sum_{i=1}^n  \sum_{j,j_1=1}^k \E X_{ij}Y_{ij_1}\frac{1}{2\sqrt{s}}\int_{\mathbb{R}^k}\bigl[\partial_{j_1} h_\epsilon(\sqrt{1-s}V_i+\sqrt{1-s}UY_i+\sqrt{s}z)\nn \\
&\kern9em -\partial_{j_1} h_\epsilon(\sqrt{1-s}V_i+\sqrt{s}z)\bigr]\partial_{j}\phi(z)dz.
}
Introducing another independent uniform random variable $U'$ in $[0,1]$ and using the integration by parts formula,
\bes{
R_2(s)
&=\sum_{i=1}^n \sum_{j=1}^k \E UX_{ij} \frac{\sqrt{1-s}}{2\sqrt{s}} \\
&\quad \times \int_{\mathbb{R}^k} ( \sum_{j_1=1}^k Y_{ij_1}\partial_{j_1} \nabla h_\epsilon (\sqrt{1-s}V_i+\sqrt{1-s} UU'Y_i+\sqrt{s}z) \cdot Y_i) \partial_j \phi(z) dz . 
}
By (\ref{h epsilon-2}) and \eq{r2.10-3},
\bes{
&|\E ^{U,U',X_{A_i}}(\sum_{j_1=1}^k Y_{ij_1} \partial_{j_1}\nabla h_\epsilon (\sqrt{1-s}V_i+\sqrt{s}z+\sqrt{1-s}UU'Y_i)\cdot Y_i)| \\
&\le \frac{8}{\epsilon^2} |Y_i|^2 \E ^{U,U', X_{A_i}}I(\sqrt{1-s}V_i \in A^{\epsilon+\delta}\backslash A^{\delta}-(\sqrt{s}z+\sqrt{1-s}UU'Y_i))\nn \\
&\le  \frac{c |Y_i|^2}{\epsilon^2 \sqrt{1-s}} \left(\frac{\theta_1^3 \sqrt{\theta_2}}{\alpha_2^{3/2}}(1+\sqrt{\frac{\beta}{\alpha_2}})\gamma+\frac{\sqrt{\theta_2 \beta}}{\alpha_2}\epsilon \right)
}
Therefore,
\bes{
|\int_0^{\epsilon^2}R_2(s)ds|&\le \frac{c}{\epsilon^2}\sum_{i=1}^n  \E  |Y_i|^2 \int_0^{\epsilon^2} \frac{1}{\sqrt{s}}\left(\frac{\theta_1^3 \sqrt{\theta_2}}{\alpha_2^{3/2}}(1+\sqrt{\frac{\beta}{\alpha_2}})\gamma+\frac{\sqrt{\theta_2 \beta}}{\alpha_2}\epsilon \right) \\
&\quad \times\int_{\mathbb{R}^k}|\sum_{j=1}^k X_{ij} \partial_j\phi(z)   |dzds \\
&\le \frac{c}{\epsilon} \sum_{i=1}^n \E |X_i||Y_i|^2 \left(\frac{\theta_1^3 \sqrt{\theta_2}}{\alpha_2^{3/2}}(1+\sqrt{\frac{\beta}{\alpha_2}})\gamma+\frac{\sqrt{\theta_2 \beta}}{\alpha_2}\epsilon \right)
}
where we used \eq{l3.3-1}.

For the case $\epsilon^2<s\leq 1$, let
$U, U'$ be independent uniform random variables in $[0,1]$. Let $U_i=\sum_{i_1\notin B_i}X_{i_1}$, and let $\{Y_i', Z_i'\}$ be independent samples of $\{Y_i, Z_i\}$ given $\{X_{i_1}: i_1\notin B_i\}$. Using the first expression of $\partial_{jj_1}g$ in \eq{Dg} and the integration by parts formula,
\bes{
&R_2(s)=\sum_{i=1}^n \sum_{j,j_1=1}^k \E X_{ij}Y_{ij_1}(-\frac{1}{2s})\int_{\mathbb{R}^k}\bigl[h_\epsilon(\sqrt{1-s}V_i+\sqrt{1-s}UY_i+\sqrt{s}z)\nn \\
&\kern9em -h_\epsilon(\sqrt{1-s}V_i+\sqrt{s}z)\bigr]\partial_{jj_1}\phi(z)dz\\
&=\sum_{i=1}^n \sum_{j,j_1,j_2=1}^k \E U X_{ij} Y_{ij_1} Y_{ij_2}  
\frac{\sqrt{1-s}}{2s^{3/2}} \\
&\kern7em \times \int_{\mathbb{R}^k} \big[ 
h_\epsilon (\sqrt{1-s}V_i+\sqrt{1-s}UU'Y_i+\sqrt{s}z) \\
&\kern10em-h_\epsilon(\sqrt{1-s}(U_i+Z_i'+Y_i')+\sqrt{s}z) \big] \partial_{jj_1j_2}\phi(z)dz\\
&\quad +\sum_{i=1}^n \sum_{j,j_1,j_2=1}^k \E U X_{ij} Y_{ij_1} Y_{ij_2}  
\frac{\sqrt{1-s}}{2s^{3/2}} \\
&\quad \times \int_{\mathbb{R}^k} \big[ 
h_\epsilon (\sqrt{1-s}(U_i+Z_i'+Y_i') +\sqrt{s}z) -h_\epsilon(\sqrt{1-s}Z +\sqrt{s}z) \big] \partial_{jj_1j_2}\phi(z)dz\\
&\quad +\sum_{i=1}^n \sum_{j,j_1,j_2=1}^k \E U X_{ij} Y_{ij_1} Y_{ij_2} 
\frac{\sqrt{1-s}}{2s^{3/2}} \int_{\mathbb{R}^k} 
h_\epsilon (\sqrt{1-s}Z+\sqrt{s}z) \partial_{jj_1j_2}\phi(z)dz\\
&=: R_{2,0}(s)+R_{2,1}(s)+R_{2,2}(s)
}
where $Z$ is an independent $k$-dimensional Gaussian vector.

Define $\eta:=\sup_{A\in\mathcal{A}} |\P(W\in A)-\P(Z\in A)|$.
By the local dependence assumption, $X_{A_i}$ is independent of $U_i+Z_i'+Y_i'$, which has the same distribution as $W$. Therefore,
\bes{
&\left|\E^{X_{A_i}} h_\epsilon\big(\sqrt{1-s}(U_i+Z_i'+Y_i')+\sqrt{s}z  \big)
-\E h(\sqrt{1-s}Z+\sqrt{s}z)   \right|\\
&\leq \P(\sqrt{1-s}Z+\sqrt{s}z\in A^{\epsilon+\delta_2}\backslash A^\delta_2)+ \sup_{A\in\mathcal{A}} |\P(W\in A)-\P(Z\in A)|\\
&\leq k^{1/2}\frac{1}{\sqrt{1-s}}\epsilon+2\eta,
}
and
\be{
|\int_{\epsilon^2}^1 R_{2,1}(s)ds|\leq c(\eta+k^{1/2}\epsilon) \frac{1}{\epsilon} \sum_{i=1}^n \E |X_i||Y_i|^2
}
by \eq{l3.4-1}.
Similarly as in proving \eq{R212'},
\be{
|\int_{\epsilon^2}^1 R_{2,2}(s)ds|\leq c\sum_{i=1}^n\E |X_i||Y_i|^2.
}
By \eq{h epsilon-2} and the fact that $Y_i', Z_i'$ are independent of $\{X_{i_1}:i_1\notin C_i\}$,
\bes{
&\left|\E^{U,U',X_{A_i}} \big[ h_\epsilon(\sqrt{1-s}(V_i+UU'Y_i)+\sqrt{s}z) 
-h_\epsilon(\sqrt{1-s}(U_i+Z_i'+Y_i')+\sqrt{s}z)  \big] \right|\\
&\leq \frac{c}{\epsilon}\E^{U,U',X_{A_i}} (|Y_i|+|Z_i|+|Y_i'|+|Z_i'|)\\
&\quad \times \E^{U,U',U'',X_{B_i}, Z_i', Y_i'} I(\sqrt{1-s}W+F\in A^{\epsilon+\delta_1}\backslash A^{\delta_2})
}
where $U''$ is an independent random variable in $[0,1]$ appeared when writing $h_\epsilon(a)-h_\epsilon(b)=(a-b)\E h_\epsilon'(U''a+(1-U'')b)$, and where $F$ is a random variable measurable with respect to $\sigma(X_{B_i}, Z_i', Y_i', U,U',U'')$.
The last inequality, along with \eq{r2.10-3} and \eq{l3.4-1}, yield
\be{
|\int_{\epsilon^2}^1 R_{2,0}(s)ds|\leq \frac{c}{\epsilon^2} \Big( \frac{\theta_1^3\sqrt{\theta_2}}{\alpha_2^{3/2}}\sqrt{\frac{\beta}{\alpha_2}}\gamma+\frac{\sqrt{\theta_2 \beta}}{\alpha_2}\epsilon \Big) 
\sum_{i=1}^n \E |X_i||Y_i|^2(|Y_i|+|Y_i'|+|Z_i|+|Z_i'|).
}

Let $\widetilde{X}_{ij}$ and $\widetilde{Y}_{ij}$ be independent copies of $X_{ij}$ and $Y_{ij}$ respectively. By the (LD2) assumption and $\E WW^t=I_{k\times k}$, we can write $R_1(s)$ as
\bes{
R_1(s)&=\sum_{i=1}^n \sum_{j,j_1=1}^k \E \widetilde{X}_{ij}\widetilde{Y}_{ij}
 \big[\partial_{jj_1} g_\epsilon (W,s) -\partial_{jj_1} g_\epsilon (U_i,s)  \big]\\
&\quad -\sum_{i=1}^n \sum_{j,j_1=1}^k \E X_{ij} Y_{ij}
 \big[\partial_{jj_1} g_\epsilon(V_i,s)-\partial_{jj_1}g_\epsilon(U_i,s)  \big].
}
Bounding $\int_0^1 R_1(s)ds$ can be done similarly as for $\int_0^1 R_2 (s) ds$. The only differences are that those $|X_i||Y_i|^2$ appearing in the upper bound for $\int_0^1 R_1(s)ds$ are changed to $(|X_i||Y_i|+\E|X_i||Y_i|)(|Y_i|+|Z_i|)$ and that those $|Y_i'|,|Z_i|'$ are changed to $|Y_i''|,|Z_i''|,|T_i''|$.

By the above bounds and the smoothing inequality \eq{smooth ineq-4}, we obtain
\ben{\label{e4.1-7}
\eta\leq c\Big\{k^{1/2}(\epsilon+\delta_2)+(\eta+k^{1/2}\epsilon)\frac{1}{\epsilon}b_1 +\Big(  \frac{\theta_1^3\sqrt{\theta_2}}{\alpha_2^{3/2}}\sqrt{\frac{\beta}{\alpha_2}}\gamma+\frac{\sqrt{\theta_2 \beta}}{\alpha_2}\epsilon  \Big) (\frac{b_1}{\epsilon}+\frac{b_2}{\epsilon^2}) \Big\}.
}
The bound \eq{t4.2-1} is obtained by choosing $\epsilon=2c b_1+\sqrt{b_2}$ and solving for $\eta$ in the above inequality.
\end{proof}

\begin{proof}[Proof of Theorem \ref{t4.1}]
Since the $\theta_1'$ in \eq{t4.1-1} can be made larger than $\theta_1$ in \eq{c3.3--1}, the conditional concentration inequality in \eq{r2.10-2} is valid with $\theta_1$ replaced by $\theta_1'$. 
Let $\delta_1=16\theta_1'^3\gamma/\alpha_1$. 
Proceed as in the proof of Theorem \ref{t2}. For a convex set $A\in \mathcal{A}$ and $\epsilon>0$, we define $h_{A,\epsilon}(w)=
\psi(d(w,A^{\delta_1})/\epsilon)$ as in \eq{h epsilon}. Let $g_{A,\epsilon}$ be defined as in \eq{g epsilon}
with $h$ replaced by $h_{A, \epsilon}$. Then $f_{A,\epsilon}(w)=\int_0^1 g_{A,\epsilon}(w,s)ds$ solves the Stein equation \eq{Stein equation 2}.
In what follows, we keep the dependence on $A$ implicit and write $h_\epsilon=h_{A,\epsilon}$, $g_\epsilon=g_{A,\epsilon}$.
We have the following smoothing inequality which is proved as for \eq{smooth ineq-2}:
\ben{\label{smooth ineq-3}
\sup_{A\in \mathcal{A}}|\P (W\in A)-\P (Z\in A)| \le \sup_{h=I_{A^{\delta_1}}: A\in\mathcal{A}}|\E h_\epsilon(W)-\E h_\epsilon(Z)|+k^{1/2}(\epsilon+\delta_1),
}
where $h_\epsilon$ is defined in (\ref{h epsilon}).
From the (LD3) assumption, with
\be{
Y_i:=\sum_{i_1\in A_i} X_{i_1},\quad V_i:=W-Y_i,
}
we have
\bes{
& \E \triangle g_\epsilon (W,s)-\E  W\cdot \nabla g_\epsilon (W,s) \\
&=\Big[ \E \triangle g_\epsilon (W,s)-\sum_{i=1}^n \E X_i\cdot (\text{Hess} g_\epsilon(V_i,s)Y_i)\Big] \\
&\quad -\sum_{i=1}^n \E X_i\cdot(\nabla g_\epsilon (W,s)-\nabla g_\epsilon (V_i,s)-\text{Hess} g_\epsilon (V_i,s)Y_i) \\
&=:R_1(s)-R_2(s).
}
By \eq{Stein equation 2}, 
\be{
\E h_\epsilon(W) -\E h_\epsilon(Z)=\int_0^1 (R_1(s)-R_2(s))ds.
}
By the proof of Theorem \ref{t4.2},
\bes{
|\int_0^{\epsilon^2} R_2(s) ds|
&\le \frac{c}{\epsilon}\sum_{i=1}^n \E|X_i||Y_i|^2 \left(\frac{\theta_1'^3 \sqrt{\theta_2}}{\alpha_1^{3/2}}(1+\sqrt{\frac{\beta}{\alpha_1}})\gamma+\frac{\sqrt{\theta_2 \beta}}{\alpha_1}\epsilon \right)\\
&\le \frac{c\theta_1'^3 \gamma}{\epsilon} \left(\frac{\theta_1'^3 \sqrt{\theta_2}}{\alpha_1^{3/2}}(1+\sqrt{\frac{\beta}{\alpha_1}})\gamma+\frac{\sqrt{\theta_2 \beta}}{\alpha_1}\epsilon \right).
}
where we used $\E|X_i||Y_i|^2\leq\E|X_i|^3+2\E|Y_i|^3/3$ and
\be{
\sum_{i=1}^n \E|Y_i|^3\leq \theta_1'^2\sum_{i=1}^n \sum_{i_1\in A_i}\E|X_{i_1}|^3\leq \theta_1'^3 \gamma.
}
Using the first expression of $\partial_{jj_1}g(w,s)$ in \eq{Dg},
\bes{
R_2(s)=&\sum_{i=1}^n \sum_{j,j_1=1}^k \E X_{ij}Y_{ij_1}(-\frac{1}{2s})\int_{\mathbb{R}^k}\bigl[h_\epsilon(\sqrt{1-s}V_i+\sqrt{1-s}UY_i+\sqrt{s}z)\nn \\
&\kern9em -h_\epsilon(\sqrt{1-s}V_i+\sqrt{s}z)\bigr]\partial_{jj_1}\phi(z)dz.
}
By \eq{h epsilon-2},
\bes{
&\left| \E^{U,X_{A_i}} \big[  h_\epsilon(\sqrt{1-s}V_i+\sqrt{1-s}UY_i+\sqrt{s}z)-h_\epsilon(\sqrt{1-s}V_i+\sqrt{s}z) \big]  \right|  \\
&\leq\frac{c}{\epsilon}|Y_i|\E^{U,X_{A_i}}\E^{U,U',X_{A_i}} I(\sqrt{1-s}(W-Y_i+UU'Y_i)+\sqrt{s}z
\in A^{\epsilon+\delta_1}\backslash A^{\delta_1})
}
where $U'$ is an independent random variable in $[0,1]$ appeared when writing $h_\epsilon(a)-h_\epsilon(b)=(a-b)\E h_\epsilon'(U'a+(1-U')b)$.
Therefore, by \eq{r2.10-2} and $\int_{\mathbb{R}^k}|\sum_{j,j_1=1}^k X_{ij} Y_{ij_1} \partial_{jj_1}\phi(z)|dz\le c|X_i||Y_i|$ by a similar argument as for \eq{l3.4-1}, we have
\be{
|\int_{\epsilon^2}^1 R_2(s)ds|\leq \frac{c\theta_1'^3 \gamma|\log \epsilon|}{\epsilon} \left(\frac{\theta_1'^3 \sqrt{\theta_2}}{\alpha_1^{3/2}}(1+\sqrt{\frac{\beta}{\alpha_1}})\gamma+\frac{\sqrt{\theta_2 \beta}}{\alpha_1}\epsilon \right).
}

Let $\widetilde{X}_{ij}$ and $\widetilde{Y}_{ij_1}$ be independent copies of $X_{ij}$ and $Y_{ij_1}$ respectively for each $i,j,j_1$, and let $U_i=W-\sum_{i_1\in B_i}X_{i_1}$. We have by the (LD2) assumption and $\E WW^t=I_{k\times k}$,
\bes{
R_1(s)=&\sum_{i=1}^n \sum_{j,j_1=1}^k \E \widetilde{X}_{ij} \widetilde{Y}_{ij_1} [\partial_{jj_1}g_\epsilon(W,s)-\partial_{jj_1} g_\epsilon(U_i,s)]   \\
&-\sum_{i=1}^n \sum_{j,j_1=1}^k \E X_{ij} Y_{ij_1} [\partial_{jj_1}g_\epsilon(V_i,s)-\partial_{jj_1} g_\epsilon(U_i,s)].
}
By the same argument as for $R_2(s)$,
\be{
|R_1|\leq \frac{c\theta_1'^3 \gamma|\log \epsilon|}{\epsilon} \left(\frac{\theta_1'^3 \sqrt{\theta_2}}{\alpha_1^{3/2}}(1+\sqrt{\frac{\beta}{\alpha_1}})\gamma+\frac{\sqrt{\theta_2 \beta}}{\alpha_1}\epsilon \right).
}
From the bounds on $R_1, R_2$ and the smoothing inequality \eq{smooth ineq-3},
\bes{
&\sup_{A\in \mathcal{A}}|\P (W\in A)-\P (Z\in A)|\\
&\leq c\left[k^{1/2}(\epsilon+\delta) +  \frac{\theta_1'^3 \gamma|\log \epsilon|}{\epsilon} \left(\frac{\theta_1'^3 \sqrt{\theta_2}}{\alpha_1^{3/2}}(1+\sqrt{\frac{\beta}{\alpha_1}})\gamma+\frac{\sqrt{\theta_2 \beta}}{\alpha_1}\epsilon \right)  \right].
}
The bound \eq{t4.1-2} is proven by choosing $\epsilon=\theta_1'^3 \gamma$.
\end{proof}

\section{Proofs of lemmas}

We prove Lemma \ref{lem1} to \ref{lem4} in this section.

\textbf{Proof of Lemma \ref{lem1}.}
The lemma is true by observing that for $x\in \mathbb{R}^k\backslash A^{\epsilon}$, $x_0$ must be the nearest point of $x_1$ in $\bar{A}$ where $x_0, x_1$ as defined above Lemma \ref{lem1}.

\textbf{Proof of Lemma \ref{lem3}.}
Because $x_0$, the nearist point in $\bar{A}$ from $x$, depends on $x$, the validity of (\ref{lem1-2}) is not obvious. We consider the following three cases. All the other cases can be reduced to these cases.

Case 1: $\eta \in \bar{A}$, $\eta +\xi\in \bar{A}$.

Case 2: $\eta \in A^{\epsilon}\backslash \bar{A}$, $\eta +\xi\in A^{\epsilon}\backslash \bar{A}$.

Case 3: $\eta \in \mathbb{R}^k\backslash A^{\epsilon}$, $\eta +\xi\in \mathbb{R}^k\backslash A^{\epsilon}$.

In case 1, since $f(\eta )=f(\eta +\xi)=0$, (\ref{lem1-2}) is satisfied.

From the facts that (\ref{lem1-2}) is equivalent to
\ben{
(-\xi) \cdot (f(\eta+\xi+(-\xi))-f(\eta+\xi)) \ge 0
}
and
\ben{\label{lem2-0}
\xi\cdot (\eta-\eta_0)>0\quad \text{implies}\quad (-\xi) \cdot ((\eta+\xi)-(\eta+\xi)_0  ) <0,
}
which can be proved using a similar argument as in the next paragraph, we only need to consider the following situation in case 2.

\ignore{
If $\xi\cdot (\eta-\eta_0)=0$. Let $p_1$ be the $(k-1)$-dimensional hyperplane orthogonal to $\eta-\eta_0$ and containing $\eta_0$. Let  $p_2$ be the $(k-1)$-dimensional hyperplane orthogonal to $\xi$ and containing $\eta+\xi$. The hyperplane $p_1$ divides $\mathbb{R}^k$ into two parts $s_1, s_2$, where $s_2$ is open and contains $\eta$; the hyperplane $p_2$ divides $\mathbb{R}^k$ into two parts $s_3, s_4$ where $s_3$ is closed and contains $\eta$. Then $(\eta+\xi)_0\in s_1\cap s_3$, which implies (\ref{lem1-2}).

If $\xi\cdot (\eta-\eta_0)>0$. Let $p_1$ be the plane parallel to $\eta-\eta_0$, $\xi$ and containing $\eta$. Let $(\eta+\xi)'$ be on $p_1$ such that $(\eta+\xi)'-(\eta+\xi)$ is parallel to $\eta-\eta_0$ and $(\eta+\xi)'-\eta_0$ is parallel to $\xi$. Let $p_2$ be the $(k-1)$-dimensional hyperplane orthogonal to $\xi$ and containing $(\eta+\xi)'$. Then $(\eta+\xi)_0$ must be on the same side of $p_2$ as $\eta_0$, which implies (\ref{lem1-2}).
}

Assume $\xi\cdot (\eta-\eta_0)\le 0$. Let $p_1$ be the plane containing points $\eta_0, \eta, \eta+\xi$. Let the point $(\eta+\xi)'$ be on $p_1$ such that $(\eta+\xi)'-(\eta+\xi)$ is parallel to $\eta_0-\eta$ and $(\eta+\xi)'-\eta_0$ is parallel to $\xi$. Let $p_2$ be the $(k-1)$-dimensional hyperplane orthogonal to $\xi$ and containing $(\eta+\xi)'$. The hyperplane $p_2$ divides $\mathbb{R}^k$ into two parts $s_1, s_2$ where $s_1$ is closed and contains $\eta$. If $(\eta+\xi)_0$, the nearest point in $\bar{A}$ from $\eta+\xi$, is in $s_1$, (\ref{lem1-2}) is satisfied. If not, let $(\eta+\xi)''$ be the projection of $(\eta+\xi)_0$ on $p_1$. Then the angle between $\eta_0-(\eta+\xi)''$ and $\eta+\xi-(\eta+\xi)''$ is less than $\pi/2$. This means that the angle between $\eta_0-(\eta+\xi)_0$ and $\eta+\xi-(\eta+\xi)_0$ is less than $\pi/2$, which contradicts with the fact that $(\eta+\xi)_0$ is the nearest point in $\bar{A}$ from $\eta+\xi$.

The validity of (\ref{lem1-2}) in case 3 can be proved similarly.

\ignore{
Next we consider Case 3. Only one situation is not trivial, that is, when $\xi\cdot (\eta-\eta_0)>0$. Let $p_1$ be the plane parallel to $\eta-\eta_0$, $\xi$ and containing $\eta$. Let $(\eta+\xi)'$ be on $p_1$ such that
\beq
(\eta+\xi)-(\eta+\xi)'=[\epsilon+(m_1+m_2)\delta] \frac{\eta-\eta_0}{|\eta-\eta_0|}.
\eeq
Let $(\eta+\xi)'''$ be in the same line as $(\eta+\xi), (\eta+\xi)'$ such that $(\eta+\xi)'''-\eta_0$ is orthogonal to $\eta-\eta_0$. Let $p_2$ be the $(k-1)$-dimensional hyperplane orthogonal to $\xi$ and containing $(\eta+\xi)'$. Define $(\eta+\xi)''$ to be the intersection of $p_2$ and the line $\{(\eta+\xi)_0+t [(\eta+\xi)-(\eta+\xi)_0]: 0 <t <1\}$. Then, if $|(\eta+\xi)-(\eta+\xi)''| \le \epsilon+ (m_1+m_2)\delta$, (\ref{lem1-2}) is satisfied. If $|(\eta+\xi)-(\eta+\xi)''| > \epsilon+ (m_1+m_2)\delta$, then the angle between $(\eta+\xi)_0-(\eta+\xi)$ and $\eta_0-(\eta+\xi)$ is bigger than the angle between $(\eta+\xi)'''-(\eta+\xi)$ and $\eta_0-(\eta+\xi)$. Also note that the angle between $(\eta+\xi)_0-\eta_0$ and $(\eta+\xi)-\eta_0$ is bigger than the angle between $(\eta+\xi)'''-\eta_0$ and $(\eta+\xi)-\eta_0$. Therefore, the angle between $\eta_0-(\eta+\xi)_0$ and $(\eta+\xi)-(\eta+\xi)_0$ is smaller than $\pi/2$, which contradicts with the fact that $(\eta+\xi)_0$ is the nearist point. Therefore, the lemma is proved.
}

\textbf{Proof of Lemma \ref{lem2}.}
We first prove $f_i$ is $1$-Lipschitz in direction $i$. From (\ref{lem2-0}), we only need to prove
\ben{\label{lem3-0}
|f_i (x+he_i) -f_i(x)| \le h, \quad h>0
}
in the following two cases.

Case 1: $x, x+he_i \in A^{\epsilon }\backslash \bar{A}$ and $e_i \cdot (x-x_0) \le 0$.

Case 2: $x, x+he_i \notin A^{\epsilon }$ and $e_i \cdot (x-x_0) \le 0$.

For case 1, let $p_1$ be the plane parallel to $x-x_0$, $e_i$ and containing $x$. Let $(x+he_i)'$ be on $p_1$ such that $(x+he_i)'-(x+he_i)$ is parallel to $x-x_0$ and $(x+he_i)'-x_0$ is parallel to $e_i$.
Let $p_2$ be the $(k-1)$-dimensional hyperplane orthogonal to $e_i$ and containing $(x+he_i)'$, and let $p_3$ be the $(k-1)$-dimensional hyperplane orthogonal to $x-x_0$ and containing $x_0$.
Let $(x+he_i)''$ be the projection of $x+he_i$ on $p_3$ and, let $x'$ be the intersection of the line $\{x_0+t(x-x_0): t\in\mathbb{R}\}$ with $p_2$. Then, $(x+h e_i)_0'$, the projection of $(x+he_i)_0$ on $p_1$, must be within the trapezoid $\{x_0, x', (x+he_i)', (x+he_i)''\}$ (including the boundary), which implies $h\ge f_i(x+h e_i) -f_i (x) \ge 0$. Therefore, (\ref{lem3-0}) is satisfied. Case 2 is similar.

Since $f_i$ is $1$-Lipschitz in direction $i$, $\partial_i f_i$ exist a.e.. From Lemma \ref{lem3},
\be{
\frac{f_i(x+he_i)-f_i(x)}{h}=\frac{(he_i)\cdot (f(x+he_i)-f(x))}{h^2} \ge 0, \forall
 \  h\in \mathbb{R}, h\ne 0.
}
Therefore,
\be{
\partial_i f_i (x)=\lim_{h\rightarrow 0} \frac{f_i(x+he_i)-f_i(x)}{h} \ge 0 \quad \text{a.e.}
}

\textbf{Proof of Lemma \ref{lem4}.} If $\theta_i=0$, $f_i(x)=x-x_0=x_i-x_{0i}$. Note that $x_0$ does not change by moving $x$ a little in the direction of $e_i$. So $\partial_i f_i (x)=1=\cos ^2 \theta_i$.

If $\theta_i=\pi/2$, Lemma \ref{lem4} follows from Lemma \ref{lem2}.

If $0< \theta_i< \pi/2$ and $h>0$ small enough such that $x+he_i\in (A^{\epsilon})^o \backslash \bar{A}$. Let $p_1$ be the $(k-1)$-dimensional hyperplane orthogonal to $x-x_0$ which contains $x_0$. Let $(x+he_i)'$ be the projection of $x+he_i$ on $p_1$. Let $p_2$ be the $(k-1)$-dimensional hyperplane orthogonal to $x_0-(x+he_i)'$ which contains $(x+he_i)'$. The hyperplane $p_1$ divides $\mathbb{R}^k$ into two parts $s_1, s_2$ where $s_2$ is open and contains $x$; the hyperplane $p_2$ divides $\mathbb{R}^k$ into two parts $s_3, s_4$ where $s_3$ is closed and contains $x$.
By observing
\be{
(x+he_i-(x+he_i)')\cdot e_i =f_i(x) +\cos ^2 \theta_i h
}
and $(x+he_i)_0$ must be in $s_1\cap s_3$, we have,
\be{
f_i(x+he_i)\ge (x+he_i-(x+he_i)')\cdot e_i=f_i(x)+\cos ^2 \theta_i h.
} 
This implies
\ben{\label{3.1}
\frac{f_i(x+he_i)-f_i(x)}{h} \ge \cos^2 \theta_i.
}
Therefore,
\be{
\lim_{h\rightarrow 0^{+}} \frac{f_i(x+he_i)-f_i(x)}{h} \ge \cos ^2 \theta_i \quad \text{a.e.}
}
So $\partial_i f_i(x)\ge \cos^2 \theta_i$ a.e. .
For the other possible choices of $\theta_i$, the arguments are similar. This completes the proof of Lemma \ref{lem4}.

\section*{Acknowledgement}

We thank a referee for encouraging us to extend our approach beyond independence.
Both authors were partially supported by Grant C-389-000-010-101 and Grant C-389-000-012-101 at the National University of Singapore. Part of the revision was done when XF was visiting Stanford University supported by NUS-Overseas Postdoctoral Fellowship from the National University of Singapore.

\end{document}